\documentclass[11pt]{article}
\usepackage{amssymb,amsmath,amsfonts,amsthm,epsfig,latexsym,color}
\usepackage{amsfonts,amsmath,amsthm, amssymb, amscd, mathrsfs}
\usepackage{latexsym, euscript, epic, eepic}
\usepackage{graphicx}
\usepackage{verbatim}
\usepackage{fourier}
\usepackage{tikz}

\usepackage[colorlinks,linkcolor=black,anchorcolor=black,citecolor=black,hyperindex=true,CJKbookmarks=true]{hyperref}

\makeatletter
\newcommand{\subsectionruninhead}{\@startsection{subsection}{2}{0mm}
{-\baselineskip}{-0mm}{\bf\large}}
\newcommand{\subsubsectionruninhead}{\@startsection{subsubsection}{3}{0mm}
{-\baselineskip}{-0mm}{\bf\normalsize}}
\makeatother

\newtheorem*{theorem*}{Theorem}

\newtheorem{theoremalph}{Theorem}

\newtheorem*{proposition*}{Proposition}
\newtheorem*{corollary*}{Corollary}
\newtheorem*{claim*}{Claim}
\newtheorem*{remark*}{Remark}
\newtheorem*{problem*}{Problem}
\newtheorem{theorem}{Theorem}[section]
\newtheorem{conjecture}{Conjecture}[section]
\newtheorem{proposition}[theorem]{Proposition}
\newtheorem{corollary}[theorem]{Corollary}
\newtheorem{lemma}[theorem]{Lemma}

\newtheorem{claim}[theorem]{Claim}
\theoremstyle{definition}
\newtheorem{definition}[theorem]{Definition}
\newtheorem{remark}[theorem]{Remark}

\newtheorem{question}[theorem]{Question}

\numberwithin{equation}{section}
\newtheorem{example}[theorem] {Example}

 \def\RR{{\mathbb R}}

    \def\cU{\mathcal{U}}

\newcommand{\orb}{\operatorname{Orb}}

\newcommand{\sing}{\operatorname{Sing}}

\newcommand{\eps}{\varepsilon}
\newcommand{\R}{\mathbb{R}}
\newcommand{\N}{{\mathbb{N}}}
\newcommand{\Z}{\mathbb{Z}}
\newcommand{\set}[1]{\left\{#1\right\}}

\newcommand{\htop}{h_{\text{top}}}

\newcommand{\F}{\mathfrak{F}}

\DeclareMathOperator{\cov}{cov}

\begin{document}

\title{Conditional intermediate entropy and Birkhoff average properties of hyperbolic flows}

%\author{Yi Shi\footnote{Y. Shi was partially supported by National Key R\&D Program of China (2021YFA1001900) and NSFC (12071007, 11831001, 12090015).}\, ,\,  
	
\author{{Xiaobo Hou$^\S$ and Xueting Tian$^{\dag}$}\\
	{\em\small School of Mathematical Science,  Fudan University}\\
	{\em\small Shanghai 200433, People's Republic of China}\\
	{\small $^\S$Email:20110180003@fudan.edu.cn;   $^{\dag}$Email:xuetingtian@fudan.edu.cn}\\}

\footnotetext{Key words and phrases: Hyperbolic sets; Singular hyperbolic attractor; Metric entropy; Horseshoe; Symbolic dynamics.  }
\footnotetext{AMS Review:    37C70;  37C83; 37D20; 37B10.   }

\maketitle

\begin{abstract}
	 Katok conjectured that every $C^{2}$ diffeomorphism $f$ on a Riemannian manifold has the intermediate entropy property, that is, for any constant $c \in[0, h_{top}(f))$, there exists an ergodic measure $\mu$ of $f$ satisfying $h_{\mu}(f)=c$. In this paper we consider a conditional intermediate metric entropy property and two conditional intermediate Birkhoff average properties  for flows. For a basic set $\Lambda$ of  a flow $\Phi$ and two continuous function $g,$ $h$ on $\Lambda,$  we obtain $$\mathrm{Int}\left\{h_{\mu}(\Phi):\mu\in \mathcal{M}_{erg}(\Phi,\Lambda)\text{ and }\int g d\mu=\alpha\right\}=\mathrm{Int}\left\{h_{\mu}(\Phi):\mu\in \mathcal{M}(\Phi,\Lambda) \text{ and }\int g d\mu=\alpha\right\},$$  $$\mathrm{Int}\left\{\int g d\mu:\mu\in \mathcal{M}_{erg}(\Phi,\Lambda)\text{ and }h_{\mu}(\Phi)=c\right\}=\mathrm{Int}\left\{\int g d\mu:\mu\in \mathcal{M}(\Phi,\Lambda) \text{ and }h_{\mu}(\Phi)=c\right\}$$ and $$\mathrm{Int}\left\{\int h d\mu:\mu\in \mathcal{M}_{erg}(\Phi,\Lambda)\text{ and }\int g d\mu=\alpha\right\}=\mathrm{Int}\left\{\int h d\mu:\mu\in \mathcal{M}(\Phi,\Lambda) \text{ and }\int g d\mu=\alpha\right\}$$ for any $\alpha\in \left(\inf_{\mu\in  \in \mathcal{M}(\Phi,\Lambda) }\int g d\mu,  \,  \sup_{\mu\in   \in \mathcal{M}(\Phi,\Lambda) }\int g d\mu\right)$ and any $c\in (0,h_{top}(\Lambda)).$  In this process, we establish 'multi-horseshoe' entropy-dense property and use it to get the goal combined with  conditional variational principles. We also obtain  same result for singular hyperbolic attractors.
\end{abstract}

%\tableofcontents

\section{Introduction}
\qquad In \cite{Katok} Katok showed that every $C^{1+\alpha}$ diffeomorphism $f$ in dimension $2$ has horseshoes of large entropies. This implies that the system has the intermediate entropy property, i.e., for any constant $c \in[0, h_{top}(f))$, there exists an ergodic measure $\mu$ of $f$ satisfying $h_{\mu}(f)=c,$ where $\htop(f)$ is the topological entropy of $f,$ and $h_\mu(f)$ is the metric entropy of $\mu.$ Katok believed that this holds in any dimension. In the last decade, a number of partial results on the intermediate entropy property have been obtained, see \cite{Ures2012,QuasSoo2016,Burguet2020,GSW2017,lsww,KKK2018,LiOpro2018,Sun2019,HXX}. 
\begin{conjecture}{(Katok)}
	Every $C^{2}$ diffeomorphism $f$ on a Riemannian manifold has the intermediate entropy property.
\end{conjecture}
In \cite{TWW2019} the authors consider the intermediate Birkhoff average. To be precise, they proved that if  $f: X \rightarrow X$ is a continuous map over a compact metric space $X$ with the periodic gluing orbit property, then there is an ergodic measure $\mu_{\alpha} \in \mathcal{M}(f, X)$ such that $\int g d\mu_\alpha=\alpha$  for any continuous function $g:X\to\mathbb{R}$ and   any constant $\alpha$ satisfying
$$
\inf _{\mu \in \mathcal{M}(f, X)} \int g d\mu<\alpha<\sup _{\mu \in \mathcal{M}(f, X)} \int g d\mu
$$
where $\mathcal{M}(f, X)$ is  the set of all $f$-invariant probability measures. (The author obtained same result for the asymptotically additive functions, readers can refer to \cite{TWW2019} for details.)

Imitating the relationships  between extremum and conditional extremum,  variational principle and conditional variational principle, we consider such a  question: 
\begin{question}\label{question-AC}
	Under certain restricted conditions, do the intermediate entropy property and intermediate Birkhoff average property hold?
\end{question}
In fact, certain restricted condition determines a subset of $\mathcal{M}(f, X).$ So in other words,  Question \ref{question-AC} means that given a subset $F$ of $\mathcal{M}(f, X),$ whether one has
\begin{equation}\label{equation-BD}
	\left\{h_{\mu}(f):\mu\in F\right\}=\left\{h_{\mu}(f):\mu\in F\cap \mathcal{M}_{erg}(f,X)\right\}
\end{equation}
and 
\begin{equation}\label{equation-BE}
	\left\{\int g d\mu:\mu\in F\right\}=\left\{\int g d\mu:\mu\in F\cap \mathcal{M}_{erg}(f,X)\right\}?
\end{equation}
Here $\mathcal{M}_{erg}(f,X)$ is  the set of all $f$-invariant ergodic probability measures.

It is clear that Question \ref{question-AC} can only be true under  some suitable restricted condition or some suitable $F.$ For example, if $F\cap \mathcal{M}_{erg}(f,X)=\emptyset,$ then  Question \ref{question-AC} is obviously false. When $F=\mathcal{M}(f,X),$ (\ref{equation-BD}) is true for every $C^{1+\alpha}$ diffeomorphism $f$ in dimension $2$ \cite{Katok}  and  (\ref{equation-BE}) is true for system satisfying the periodic gluing orbit property \cite{TWW2019}. In the two proof, a basic requirement is that $\mathcal{M}_{erg}(f,X)$ is dense in $\mathcal{M}(f,X).$ So 
we believe that a suitable  $F$ should satisfies the following properties:
 \begin{description}
 	\item[(1)] $\mathcal{M}_{erg}(f,X)\cap F$ is dense in $F;$
 	\item[(2)] $F$ is convex.
 \end{description}
We will see that the following two sets satisfies item (1) and (2) for hyperbolic sets:
\begin{equation*}
	F_1(c)=\left\{\mu\in \mathcal{M}(f,X):h_{\mu}(f)=c\right\},\ F_2^g(\alpha)=\left\{\mu\in \mathcal{M}(f,X):\int gd\mu=\alpha\right\},
\end{equation*}
 where $c \in[0, h_{top}(f))$ and $
 \inf _{\mu \in \mathcal{M}(f, X)} \int g d\mu<\alpha<\sup _{\mu \in \mathcal{M}(f, X)} \int g d\mu.
 $
 And for two continuous functions $g,h$ we will consider whether the following three equalities hold: a conditional intermediate entropy property 
 \begin{equation}\label{equation-BG}
 	\left\{h_{\mu}(f):\mu\in F_2^g(\alpha)\right\}=\left\{h_{\mu}(f):\mu\in F_2^g(\alpha)\cap \mathcal{M}_{erg}(f,X)\right\},
 \end{equation}
and two conditional intermediate Birkhoff average properties
 \begin{equation}\label{equation-BH}
 	\left\{\int g d\mu:\mu\in F_1(c)\right\}=\left\{\int g d\mu:\mu\in F_1(c)\cap \mathcal{M}_{erg}(f,X)\right\},
 \end{equation}
\begin{equation}\label{equation-BJ}
	\left\{\int h d\mu:\mu\in F_2^g(\alpha)\right\}=\left\{\int h d\mu:\mu\in F_2^g(\alpha)\cap \mathcal{M}_{erg}(f,X)\right\}.
\end{equation}
Denote $L_{g}=\{\int g d\mu:\mu\in \mathcal{M}(f,X)\}.$
Note that if the following two sets are equal
\begin{equation*}
	\begin{split}
		&\mathrm{Int}\left\{(\int g d\mu, h_\mu(f)):\mu\in \mathcal{M}(f,X)\right\}\\
		&\mathrm{Int}\left\{(\int g d\mu, h_\mu(f)):\mu\in \mathcal{M}_{erg}(f,X)\right\},
	\end{split}
\end{equation*} 
then (\ref{equation-BG}) and (\ref{equation-BH}) hold except extrems for any $c \in[0, h_{top}(f))$ and  $
\inf _{\mu \in \mathcal{M}(f, X)} \int g d\mu<\alpha<\sup _{\mu \in \mathcal{M}(f, X)} \int g d\mu.
$
And if the following two sets are equal
\begin{equation*}
	\begin{split}
		&\mathrm{Int}\left\{(\int g d\mu, \int h d\mu):\mu\in \mathcal{M}(f,X)\right\}\\
		&\mathrm{Int}\left\{(\int g d\mu, \int h d\mu):\mu\in \mathcal{M}_{erg}(f,X)\right\},
	\end{split}
\end{equation*} 
then (\ref{equation-BJ}) hold except extrems for any $
\inf _{\mu \in \mathcal{M}(f, X)} \int g d\mu<\alpha<\sup _{\mu \in \mathcal{M}(f, X)} \int g d\mu.
$

In this paper, we are interested in flows. We first recall some notions of flows.
Let $\mathscr{X}^r(M)$, $r\ge 1$, denote the space of $C^r$-vector fields on a compact Riemannian manifold $M$ endowed with the $C^r$ topology. For $X\in\mathscr{X}^r(M)$, denote by $\phi_t^X$ or $\phi_t$ for simplicity the $C^r$-flow generated by $X$ and denote by ${\rm D}\phi_t$ the tangent map of $\phi_t$. 
Given a vector field $X\in\mathcal{X}^1(M)$ and a compact invariant subset $\Lambda$ of  the $C^1$-flow  $\Phi=(\phi_t)_{t\in\mathbb{R}}$ generated by $X$, we denote by $C(\Lambda,\mathbb{R})$ the space of continuous functions on $\Lambda$. The set of invariant (resp. ergodic) probability measures of $X$ supported on $\Lambda$ is denoted by $\mathcal{M}(\Phi,\Lambda)$ (resp. $\mathcal{M}_{erg}(\Phi,\Lambda)$), and it is endowed with the weak$^*$-topology. We denote by $h_\mu(X)$ or $h_\mu(\Phi)$  the metric entropy of the invariant probability measure $\mu\in \mathcal{M}(\Phi,\Lambda)$, defined as the metric entropy of $\mu$ with respect to the time-1 map $\phi^X_1$ of the flow. Let $d^*$ be a translation invariant metric on the space $\mathcal{M}(\Phi,\Lambda)$
compatible with the weak$^*$ topology.
We focus on the following question for flows in the present paper.
\begin{question}\label{Conjecture-2}
	For every  flow  $(\phi_t)_{t\in\mathbb{R}}$ generated by a typical vector field $X$, and two continuous function $g,$ $h$ on a compact invariant subset $\Lambda$ of the $C^1$ flow $(\phi_t)_{t\in\mathbb{R}},$
	do the following two equalities hold:
	\begin{equation*}
			\mathrm{Int}\left\{(\int g d\mu, h_\mu(\Phi)):\mu\in \mathcal{M}(\Phi,\Lambda)\right\}
			=\mathrm{Int}\left\{(\int g d\mu, h_\mu(\Phi)):\mu\in \mathcal{M}_{erg}(\Phi,\Lambda)\right\},
	\end{equation*} 
    and
    \begin{equation*}
    		\mathrm{Int}\left\{(\int g d\mu, \int h d\mu):\mu\in \mathcal{M}(\Phi,\Lambda)\right\}
    		=\mathrm{Int}\left\{(\int g d\mu, \int h d\mu):\mu\in \mathcal{M}_{erg}(\Phi,\Lambda)\right\}?
    \end{equation*}
\end{question}

We will answer this question partially for hyperbolic flows and singular hyperbolic folws. 
\begin{definition}\label{Def:horseshoe}%\marginpar{\textcolor{red}{or suspension of SFT? Be careful of C-dense case (Bowen).}}
	Given $X\in\mathscr{X}^1(M)$, 
	an invariant compact set $\Lambda$ is called a \emph{basic set} if  it is, transitive, hyperbolic and  locally maximal, is not reduced to a single orbit of a hyperbolic critical element
	and its intersection with any local cross-section to the flow is totally disconnected.
\end{definition}
Given a continuous function $g$ on a compact invariant subset $\Lambda$ of the $C^1$ flow $(\phi_t)_{t\in\mathbb{R}},$ denote $L_{g}=\{ \int g d\mu:\mu\in \mathcal{M}(\Phi,\Lambda)\}.$
For any $\alpha\in  L_{g},$ denote  $$M_{g}(\alpha)=\left\{\mu\in \mathcal{M}(\Phi,\Lambda):\int g d\mu=\alpha\right\},\ M_{g}^{erg}(\alpha)=M_{g}(\alpha)\cap  \mathcal{M}_{erg}(\Phi,\Lambda),$$
and 
$$M_{g}^{top}(\alpha)=\left\{\mu\in M_g(\alpha):h_{\mu}(\Phi)=\sup\left\{h_{\nu}(\Phi):\nu\in M_{g}(\alpha)\right\}\right\}.$$
Then $M_{g}(\alpha)$ is  a closed subset of  $\mathcal{M}(\Phi,\Lambda),$ and thus $$\sup\left\{h_{\nu}(\Phi):\nu\in M_{g}(\alpha)\right\}=\max\left\{h_{\nu}(\Phi):\nu\in M_{g}(\alpha)\right\}$$ for any $\alpha\in L_{g}$ if the entropy function $\mathcal{M}(\Phi,\Lambda) \ni \mu \mapsto h_\mu(\Phi)$  is upper semi-continuous.
For a probability measure $\mu\in \mathcal{M}(\Phi,\Lambda)$ we denote  the support of $\mu$ by $$S_\mu:=\{x\in \Lambda:\mu(U)>0\ \text{for any neighborhood}\ U\ \text{of}\ x\}.$$
Now we state our first main result.
\begin{theoremalph}\label{Thm:basic-set1}
	Let $X\in\mathscr{X}^1(M)$ and $\Lambda$ be a basic set of $X$.  If $g$ is a continuous function on $\Lambda$,  then for any $\alpha\in \mathrm{Int}(L_{g}),$ any $\mu\in M_{g}(\alpha)\setminus M_{g}^{top}(\alpha),$ any $0\leq c\leq h_{\mu}(\Phi)$ and any $\zeta>0$, there is $\nu\in M_{g}^{erg}(\alpha)$ such that $d^*(\nu,\mu)<\zeta$ and $h_{\nu}(\Phi)=c.$ Moreover, for  any $\alpha\in \mathrm{Int}(L_{g})$ and $0\leq c< \max\{h_{\mu}(\Phi):\mu\in M_{g}(\alpha)\},$ the set $\{\mu\in M_{g}^{erg}(\alpha):h_{\mu}(\Phi)=c,\ S_\mu=\Lambda\}$ is residual in $\{\mu\in M_{g}(\alpha):h_{\mu}(\Phi)\geq c\}.$
\end{theoremalph}
\begin{remark}
	Let $X\in\mathscr{X}^1(M)$ and $\Lambda$ be a basic set of $X$. For a continuous function $g$, let $g_{t}=\int_{0}^{t} g(\phi_\tau(x))d\tau,$ then  $(g_t)_{t\geq 0}$ is a additive family of contiunous functions. So 
	let $\chi\equiv 0$ and $d=1$ in Theorem \ref{thm-almost2}(IV),  
	we have 
	\begin{equation}\label{equation-CA}
		\mathrm{Int}\left\{(\int g d\mu, h_\mu(\Phi)):\mu\in\mathcal{M}(\Phi,\Lambda)\right\}=
		\mathrm{Int}\left\{(\int g d\mu, h_\mu(\Phi)):\mu\in \mathcal{M}_{erg}(\Phi,\Lambda)\right\}.
	\end{equation} 
\begin{figure}[h]\caption{Graph of $(\int g d\mu, h_\mu(\Phi))$}\label{fig-1}
	\begin{center}
		\begin{tikzpicture}[domain=0:2]
			\draw[->] (-1,0) -- (6,0)
			node[below right] {$\int g d\mu$};
			\draw[->] (0,-1) -- (0,4)
			node[left] {$h_\mu(\Phi)$};
			\draw(0.7,2.5) .. controls (3.5,4) .. (5.2,2.2)
			node[below] at (3,2) {$(\int g d\mu, h_\mu(\Phi))$};
			\draw[dashed] (0.7,2.5) -- (0.7,0);
			\draw[dashed] (5.2,2.2) -- (5.2,0);
		\end{tikzpicture}
	\end{center}
\end{figure}
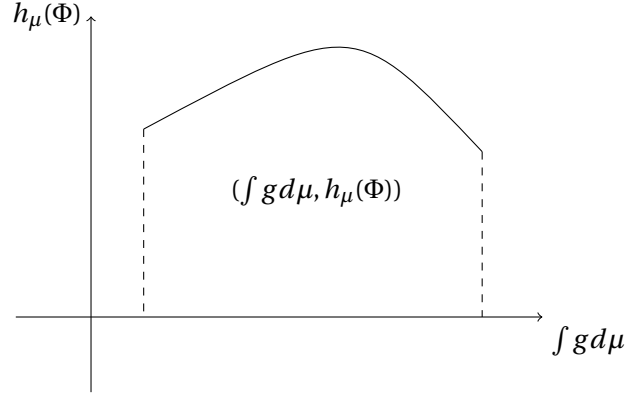
We draw the graph of $(\int g d\mu, h_\mu(\Phi))$ in Figure \ref{fig-1}. Then by (\ref{equation-CA}) every point in  interior of the region  can be attained by ergodic measures.
From (\ref{equation-CA}) , we obtain one conditional intermediate metric entropy property:
$$\mathrm{Int}\left\{h_{\mu}(\Phi):\mu\in \mathcal{M}_{erg}(\Phi,\Lambda)\text{ and }\int g d\mu=\alpha\right\}=\mathrm{Int}\left\{h_{\mu}(\Phi):\mu\in \mathcal{M}(\Phi,\Lambda) \text{ and }\int g d\mu=\alpha\right\},$$ 
and one conditional intermediate Birkhoff average property:
$$\mathrm{Int}\left\{\int g d\mu:\mu\in \mathcal{M}_{erg}(\Phi,\Lambda)\text{ and }h_{\mu}(\Phi)=c\right\}=\mathrm{Int}\left\{\int g d\mu:\mu\in \mathcal{M}(\Phi,\Lambda) \text{ and }h_{\mu}(\Phi)=c\right\}$$
for any $\alpha\in\mathrm{Int}(L_{g})$ and any $c\in (0,h_{top}(\Lambda)).$ 
\end{remark}
\begin{remark}
	Given $X \in \mathscr{X}^1(M)$ and an invariant compact set $\Lambda$. $\mathcal{M}(\Phi,\Lambda)$ is said to have entropy-dense property if for any $\varepsilon>0$ and any $\mu \in \mathcal{M}(\Phi,\Lambda)$, there exists $v \in\mathcal{M}_{erg}(\Phi,\Lambda)$ satisfying
	$$
	d^*(\mu, v)<\varepsilon \text { and } h_\nu(\Phi)>h_\mu(\Phi)-\varepsilon.
	$$
	Given a continuous function $g$ on $\Lambda,$ we say  $\mathcal{M}(\Phi,\Lambda)$ has entropy-dense property with same $g$-level if for any $\alpha\in \mathrm{Int}(L_{g}),$ any $\mu\in M_{g}(\alpha)$ and any $\varepsilon>0$, there is $\nu\in M_{g}^{erg}(\alpha)$ such that $$
	d^*(\mu, v)<\varepsilon \text { and } h_\nu(\Phi)>h_\mu(\Phi)-\varepsilon.
	$$
	By Theorem \ref{Thm:basic-set1},  if $\Lambda$ is a basic set of $X$ and $g$ is a continuous function on $\Lambda$, then $\mathcal{M}(\Phi,\Lambda)$ has entropy-dense property with same $g$-level. Moreover,  we can choose $\nu\in M_{g}^{erg}(\alpha)$ such that $d^*(\mu, v)<\varepsilon$ and $
	h_\nu(\Phi)=h_\mu(\Phi)
	$ if $\mu\in M_{g}(\alpha)\setminus M_{g}^{top}(\alpha).$
\end{remark}

Given two continuous function $g,$ $h$ on a compact invariant subset $\Lambda$ of the $C^1$ flow $(\phi_t)_{t\in\mathbb{R}},$ we denote $$L_{g,h}=\{(\int g d\mu, \int h d\mu):\mu\in \mathcal{M}(\Phi,\Lambda)\}.$$
Then $L_{g,h}$ is a nonempty convex compact subset of $\mathbb{R}^2.$
For any $\alpha\in  L_{g,h},$ denote  $$M_{g,h}(\alpha)=\left\{\mu\in \mathcal{M}(\Phi,\Lambda):(\int g d\mu, \int h d\mu)=\alpha\right\},\ M_{g,h}^{erg}(\alpha)=M_{g,h}(\alpha)\cap \mathcal{M}_{erg}(\Phi,\Lambda).$$ 
Now we state our second main result.
\begin{theoremalph}\label{Thm:basic-set2}
	Let $X\in\mathscr{X}^1(M)$ and $\Lambda$ be a basic set of $X$.  If   $g,h$ are continuous functions on $\Lambda$,  then for any $\alpha\in \mathrm{Int}(L_{g,h}),$ any $\mu\in M_{g,h}(\alpha)$  and any $\zeta>0$, there is $\nu\in M_{g,h}^{erg}(\alpha)$ such that $d^*(\nu,\mu)<\zeta.$ Moreover, for  any $\alpha\in \mathrm{Int}(L_{g,h}),$  the set $\{\mu\in M_{g,h}^{erg}(\alpha):S_\mu=\Lambda\}$ is residual in $M_{g,h}(\alpha).$
\end{theoremalph}

\begin{remark}
	Let $X\in\mathscr{X}^1(M)$ and $\Lambda$ be a basic set of $X$.  If   $g,h$ are continuous functions on $\Lambda$, then by Theorem \ref{Thm:basic-set2} we have 
	$$\mathrm{Int}\left\{(\int g d\mu, \int h d\mu):\mu\in \mathcal{M}(\Phi,\Lambda)\right\}
	=\mathrm{Int}\left\{(\int g d\mu, \int h d\mu):\mu\in \mathcal{M}_{erg}(\Phi,\Lambda)\right\}.$$
    Then we obtain a  conditional intermediate Birkhoff average property: $$\mathrm{Int}\left\{\int h d\mu:\mu\in \mathcal{M}_{erg}(\Phi,\Lambda)\text{ and }\int g d\mu=\alpha\right\}=\mathrm{Int}\left\{\int h d\mu:\mu\in \mathcal{M}(\Phi,\Lambda) \text{ and }\int g d\mu=\alpha\right\}$$ for any $\alpha\in \mathrm{Int}(L_{g}).$
\end{remark}

In the process of proving Theorem \ref{Thm:basic-set1} and \ref{Thm:basic-set2},  there are two keypoints: 'multi-horseshoe' entropy-dense property (see Theorem  \ref{Mainlemma-convex-by-horseshoe5}) and  conditional variational principles (see Theorem \ref{BarreiraDoutor2009-theorem3} and Theorem \ref{condition-principle}) proved by L. Barreira and C. Holanda in  \cite{BH21,Holanda}.

\subsection*{Outline of the paper}
In Section \ref{Section:ED-HS-approximation} we introduce 'multi-horseshoe' entropy-dense property for flows  and prove that it holds for basic sets.
In Section \ref{Almost Additive}, we recall nonadditive thermodynamic formalism for flows, 
 give  abstract conditions on which the results of Theorem \ref{Thm:basic-set1} and \ref{Thm:basic-set2}  hold in the more general context of asymptotically additive families of continuous functions.
In Section \ref{section-thm} using 'multi-horseshoe' entropy-dense property we show that the abstract conditions given in  Section \ref{Almost Additive} are satisfied for basic sets, and thus we obtain Theorem \ref{Thm:basic-set1} and \ref{Thm:basic-set2}. In Section \ref{section-singular hyperbolic},  we consider singular hyperbolic attractors and give corresponding results on Question \ref{Conjecture-2}.

\section{'Multi-horseshoe' entropy-dense property}\label{Section:ED-HS-approximation}

In this section,   we introduce 'multi-horseshoe' entropy-dense property  and prove it holds for basic sets.
We  first recall  the definition of hyperbolic.
\begin{definition}\label{Def:hyp}
	Given a vector field $X\in\mathscr{X}^1(M)$, a compact $\phi_t$-invariant set $\Lambda$ is {\it hyperbolic} if $\Lambda$ admits a continuous ${\rm D}\phi_t$-invariant splitting $T_{\Lambda}M=E^s\oplus \langle X\rangle\oplus E^u$, where $\langle X\rangle$ denotes the one-dimensional linear space generated by the vector field, and $E^s$ (resp. $E^u$) is uniformly contracted (resp. expanded) by ${\rm D}\phi_t$, that is to say, there exist constants $C>0$ and $\eta>0$ such that for any $x\in\Lambda$ and any $t\geq 0$,
	\begin{itemize}
		\item  $\|{\rm D}\phi_t(v)\|\leq Ce^{-\eta t}\|v\|$, for any $v\in E^s(x)$; and
		\item  $\|{\rm D}\phi_{-t}(v)\|\leq Ce^{-\eta t}\|v\|$, for any $v\in E^u(x)$.
	\end{itemize}
	for any $x\in\Lambda$ and $t\geq 0$. A hyperbolic set $\Lambda$ is said to be {\it locally maximal} if  there exists an open neighborhood $U$ of $\Lambda$, such that $\Lambda=\cap_{t\in\RR}\phi_t(U)$)
\end{definition}
Now we give the definition of horseshoe.
\begin{definition}\label{definition-hyperbolic}
	Given $X\in\mathscr{X}^1(M)$ and a hyperbolic invariant compact  set $\Lambda$. We call $\Lambda$  a {\it horseshoe} of $\phi_{t},$ if there exists a suspension flow $f_{t}: \Delta_{\rho} \rightarrow \Delta_{\rho}$ over a transitive subshift of finite type $(\Delta,\sigma)$ with a Lipschitz roof function $\rho$, and a homeomorphism $\pi: \Delta_{\rho} \rightarrow \Lambda$, such that $\pi\circ f_t=\phi_t\circ \pi.$
\end{definition}
\begin{remark}
	Lipschitz continuity of $\rho$ is used in Lemma \ref{lemma-G}.
\end{remark}

For any $m\in\N$ and $\{\nu_i\}_{i=1}^m \subseteq \mathcal{M}(\Phi,\Lambda)$,   we denote the convex combination of $\{\nu_i\}_{i=1}^m$ by
$$\cov\{\nu_i\}_{i=1}^m:=\left\{\sum_{i=1}^m\theta_i\nu_i:\theta_i\in[0,  1],  1\leq i\leq m~\textrm{and}~\sum_{i=1}^m\theta_i=1\right\}.$$
For any nonempty subsets of $\mathcal{M}(\Phi,\Lambda),$ $A$ and $B,$ we denote the Hausdorff distance between them  by $$d_H(A,  B):=\max\set{\sup_{x\in A}\inf_{y\in B}d^*(x,  y),\sup_{y\in B}\inf_{x\in A}d^*(y, x)  }.$$
We denote by $\htop(\Lambda)$ the topological entropy of $\Lambda$, defined as the topological entropy of $\Lambda$ with respect to the time-1 map $\phi_1$ of the flow.
\begin{definition}\label{def-strong-basic}
	Given $X\in\mathscr{X}^1(M)$ and an invariant compact  set $\Lambda$. Let $\mathcal{N}\subset  \mathcal{M}(\Phi,\Lambda)$ be a nonempty set. We say $\Lambda$ satisfies the {\it 'multi-horseshoe' entropy-dense property on $\mathcal{N}$} (abbrev. {\it 'multi-horseshoe' dense property}) if for any $F=\cov\{\mu_i\}_{i=1}^m\subseteq \mathcal{N},$ and any $\eta,  \zeta>0$,   there exist compact invariant subsets $\Lambda_i\subseteq\Theta\subsetneq \Lambda$ such that for each $1\leq i\leq m$
	\begin{enumerate}
		\item $\Lambda_i$ and $\Theta$ are horseshoes.
		\item $\htop(\Lambda_i)>h_{\mu_i}(X)-\eta$ and consequently,   $\htop(\Theta)>\sup\{h_{\kappa}(X):\kappa\in F\}-\eta.$
		\item $d_H(F,  \mathcal{M}(\Phi,\Theta))<\zeta$,   $d_H(\mu_i,  \mathcal{M}(\Phi,\Lambda_i))<\zeta$.
	\end{enumerate}
    For convenience, when $\mathcal{N}=\mathcal{M}(\Phi,\Lambda)$ we say $\Lambda$ satisfies the 'multi-horseshoe' entropy-dense property.
\end{definition}

\begin{theorem}\label{Mainlemma-convex-by-horseshoe5}
	Let $X\in\mathscr{X}^1(M)$ and $\Lambda$ be a basic set of $X$. 
	Then $\Lambda$ satisfies the 'multi-horseshoe' dense property.
\end{theorem}

We  will prove Theorem \ref{Mainlemma-convex-by-horseshoe5} by three steps. In the process, we give various versions of 'multi-horseshoe' dense property on transitive subshifts of finite type, suspension flows over transitive subshifts of finite type and basic sets.
\subsection{Transitive subshifts of finite type}
For a homeomorphism $f:K\to K$ of a compact metric space $(K,d)$, we denote by $\mathcal{M}(f,K)$  the space of $f$-invariant probability measures. And let  $\htop(f,\Lambda)$  dente the topological entropy of an $f$-invariant compact set $\Lambda,$ and $h_{\mu}(f)$ denote the metric entropy of an $f$-invariant measure.
In \cite{DHT} a result on  'multi-horseshoe' dense property of homeomorphisms was obtained.
\begin{theorem}\label{Mainlemma-convex-by-horseshoe2}\cite[Theorem 2.6]{DHT}
	Suppose a homeomorphism $f:K\to K$ of a compact metric space $(K,d)$ is transitive, expansive and satisfies the shadowing property.   Then for any $F=\cov\{\mu_i\}_{i=1}^m\subseteq \mathcal{M}(f, K),$ any $x\in K$ and any $\eta,  \zeta,\eps>0$,   there exist compact invariant subsets $\Lambda_i\subseteq\Xi\subsetneq K$ such that for each $1\leq i\leq m$
	\begin{enumerate}
		\item $(\Lambda_i,f|_{\Lambda_i})$ and $(\Xi,f|_{\Xi})$ are transitive, expansive and satisfy the shadowing property.
		\item $\htop(f,  \Lambda_i)>h_{\mu_i}(f)-\eta$ and consequently,   $\htop(f,  \Xi)>\sup\{h_{\kappa}(f):\kappa\in F\}-\eta.$
		\item $d_H(F,  \mathcal{M}(f,  \Xi))<\zeta$,   $d_H(\mu_i,  \mathcal{M}(f,  \Lambda_i))<\zeta$.
		\item There is a positive integer $L$ such that for any $z\in\Xi$ one has  $d(f^{j+mL}(z),x)<\eps$ for some $0\leq j\leq L-1$ and any $m\in\Z$.
	\end{enumerate}
\end{theorem}

Let $k$ be a poitive integer. Consider the two-side full symbolic space $$\Sigma=\prod_{-\infty}^{\infty} \{0, 1, \ldots, k-1\},$$ and the shift homeomorphism $\sigma: \Sigma \to \Sigma$ defined by $$(\sigma(w))_{n}=w_{n+1},$$ where $w =\left(w_{n}\right)_{-\infty}^{\infty}.$ A metric on $\Sigma$ is defined by $d(x, y)=2^{-m}$ if $m$ is the largest positive integer with $x_{n}=y_{n}$ for any $|n|<m$, and $d(x, y)=1$ if $x_{0} \neq y_{0}.$ If $\Delta$ is a closed subset of $\Sigma$ with $\sigma (\Delta)=\Delta,$ then $\left.\sigma\right|_{\Delta}: \Lambda \to \Delta$ is called a subshift. We usually write this as $\sigma: \Delta \to\Delta.$ A subshift $\sigma: \Delta\rightarrow \Delta$ is said to be of finite type if there exists some positive integer $N$ and a collection of blocks of length $N+1$ with the property that $x=\left(x_{n}\right)_{-\infty}^{\infty} \in \Delta$ if and only if each block $\left(x_{i}, \ldots, x_{i+N}\right)$ in $x$ of length $N+1$ is one of the prescribed blocks. 

Recall from \cite{Walters2} a subshift satisfies the shadowing property if and only if it is a subshift of finite type. As a subsystem of two-side full shift, it is expansive. So we have the following.
\begin{corollary}\label{Mainlemma-convex-by-horseshoe3}
	Suppose  $\sigma:\Delta\to \Delta$ is a transitive subshift of finite type.   Then for any $F=\cov\{\mu_i\}_{i=1}^m\subseteq \mathcal{M}(\sigma,  \Delta),$ any $x\in \Delta$ and any $\eta,  \zeta,\eps>0$,   there exist compact invariant subsets $\Delta_i\subseteq\Xi\subsetneq \Delta$ such that for each $1\leq i\leq m$
	\begin{enumerate}
		\item $(\Delta_i,\sigma|_{\Delta_i})$ and $(\Xi,\sigma|_{\Xi})$ are transitive subshifts of finite type.
		\item $\htop(\sigma,  \Delta_i)>h_{\mu_i}(\sigma)-\eta$ and consequently,   $\htop(\sigma,  \Xi)>\sup\{h_{\kappa}(\sigma):\kappa\in F\}-\eta.$
		\item $d_H(F,  \mathcal{M}(\sigma,  \Xi))<\zeta$,   $d_H(\mu_i,  \mathcal{M}(\sigma,  \Delta_i))<\zeta$.
		\item There is a positive integer $L$ such that for any $z\in\Xi$ one has  $d(\sigma^{j+mL}(z),x)<\eps$ for some $0\leq j\leq L-1$ and any $m\in\Z$.
	\end{enumerate}
\end{corollary}

\subsection{Suspension flows over transitive subshifts of finite type}\label{subsec:suspension}
Let $f\colon K\to K$ be a homeomorphism on a compact metric space $(K,d)$ and consider a continuous roof function $\rho\colon K\to(0,+\infty)$.
We define the suspension space to be
$$K_{\rho}=\left\{(x,s)\in K\times [0,+\infty) \colon 0\leq s\leq \rho(x)\right\}/\sim,$$
where the equivalence relation $\sim$ identifies $(x,\rho(x))$ with $(f(x),0)$, for all $x\in K$. 
Denote
$\pi$
 the quotient map from $K\times[0,+\infty)$ to $K_{\rho}$. 
We define the suspension flow over  $f\colon K\to K$ with roof function $\rho$ by
$$f_t(x,s)=\pi(x,s+t).$$
For any function $g\colon K_{\rho}\to \R$, we associate the function $\varphi_g\colon K\to\R$ by $$\displaystyle\varphi_g(x)=\int_0^{\rho(x)}g(x,t){\rm d}t.$$ Since the roof function $\rho$ is continuous, $\varphi_g$ is continuous as long as $g$ is. Moreover, to each invariant probability measure $\mu$ we associate the measure $\mu_{\rho}$ given by
$$\displaystyle\int_{K_{\rho}}g {\rm d}\mu_{\rho}=\frac{\int_K\varphi_g {\rm d}\mu}{\int_K\rho {\rm d}\mu} \qquad \forall g\in C(K_{\rho},\R). $$
Observe that not only the measure $\mu_{\rho}$ is $\F$-invariant (i.e. $\mu_{\rho}(f_t^{-1}A)=\mu_{\rho}(A)$ for all $t\geq0$ and measurable sets $A$), but also using that $\rho$ is bounded away from zero, the map 
$$
\mathcal{R}\colon\mathcal{M}(f,K)\to \mathcal{M}(\F,K_{\rho}) \quad\text{given by} \quad \mu\mapsto\mu_{\rho}
$$ 
is a bijection. Abramov's theorem \cite{Abramov,Parry-Pollicott} states that 
$\displaystyle h_{\mu_{\rho}}(\F)=h_{\mu}(f)/\int\rho {\rm d }\mu$ and hence,
the topological entropy $\htop(\F)$ of the flow satisfies
$$
\htop(\F)=\sup\Big\{h_{\mu_{\rho}}(\F)\colon\mu_{\rho}\in \mathcal{M}(\F,K_{\rho})\Big\}=\sup\left\{\frac{h_{\mu}(f)}{\int\rho {\rm d}\mu}\colon\mu\in \mathcal{M}(f,K)\right\}.
$$ 
Throughout we will use the notation $\Phi=(\phi_t)_{t}$ for a flow on a compact metric space and  
$\F=(f_t)_{t}$ for a suspension flow.

Consider a suspension flow $f_{t}: \Delta_{\rho} \rightarrow \Delta_{\rho}$ over a transitive subshift of finite type $(\Delta,\sigma)$ with a continuous roof function $\rho.$ A metric on $\Delta$ is defined by $d(x, y)=2^{-m}$ if $m$ is the largest positive integer with $x_{n}=y_{n}$ for any $|n|<m$, and $d(x, y)=1$ if $x_{0} \neq y_{0}.$ There is a natural metric $d_{\Delta_\rho}$, known as the Bowen-Walters metric and we have $d_{\Delta_\rho}((x,0),(y,0))=d(x,y)$ for any $x,y\in\Delta$ by \cite{BS00b}.  Now we state a result on 'multi-horseshoe' dense property of suspension flows.
\begin{proposition}\label{Mainlemma-convex-by-horseshoe4}
	Suppose  $f_{t}: \Delta_{\rho} \rightarrow \Delta_{\rho}$ is a suspension flow over a transitive subshift of finite type $(\Delta,\sigma)$ with a continuous roof function $\rho.$   Then for any $F=\cov\{\mu^i_\rho\}_{i=1}^m\subseteq \mathcal{M}(\F,  \Delta_\rho),$ any $x\in \Delta$ and any $\eta,  \zeta,\eps>0$,   there exist compact $\F$-invariant subsets $\Delta^i_\rho\subseteq\Xi_\rho\subsetneq \Delta_\rho$ such that for each $1\leq i\leq m$
	\begin{enumerate}
		\item $f_{t}|_{\Delta^i_\rho}: \Delta^i_\rho \rightarrow \Delta^i_\rho$  and $f_{t}|_{\Xi_\rho}: \Xi_\rho\rightarrow \Xi_\rho$ are suspension flows over transitive subshifts of finite type $(\Delta_i,\sigma|_{\Delta_i})$ and $(\Xi,\sigma|_{\Xi})$  with the roof function $\rho.$ 
		\item $\htop(\Delta^i_\rho)>h_{\mu_\rho^i}(\F)-\eta$ and consequently,   $\htop(\Xi_\rho)>\sup\{h_{\kappa}(\F):\kappa\in F\}-\eta.$
		\item $d_H(F,  \mathcal{M}(\F,  \Xi_\rho))<\zeta$,   $d_H(\mu_\rho^i,  \mathcal{M}(\F,  \Delta^i_\rho))<\zeta$.
		\item For any $z\in\Xi$ there is an increasing sequence $\{t_m\}_{m=-\infty}^{+\infty}$ such that $f_{t_m}((z,0))\in \Delta\times \{0\},$  $d_{\Delta_{\rho}}(f_{t_m}((z,0)),(x,0))<\eps$ for any $m\in\Z,$ and $\lim\limits_{m\to+\infty}t_m=\lim\limits_{m\to-\infty}-t_m=+\infty$.
	\end{enumerate}
\end{proposition}
\begin{proof}
	Each $\F$-invariant probability measure $\mu_\rho^i$ is determined by a $\sigma$-invariant probability measure $\mu_i$. 
	Take $\tilde{\zeta},\tilde{\eta}>0$ small enough such that if $\mu,\nu\in\mathcal{M}(\sigma,\Delta)$ satisfy $d^*(\mu,\nu)<\tilde{\zeta},$ then one has 
	\begin{equation}\label{equation-AA}
		d^*(\mathcal{R}(\mu),\mathcal{R}(\nu))<\zeta \text{ and }\frac{\int \rho d\mu}{\int\rho d\nu}(h_{\mathcal{R}(\mu)}(\F)-\frac{2\tilde{\eta}}{\int\rho d \mu})>h_{\mathcal{R}(\mu)}(\F)-\eta.
	\end{equation}
    For the $\tilde{F}=\cov\{\mu_i\}_{i=1}^m\subseteq \mathcal{M}(\sigma,\Delta),$ $x\in \Delta$ and  $\tilde{\eta},  \tilde{\zeta},\eps>0,$ by Corollary \ref{Mainlemma-convex-by-horseshoe3} there exist compact invariant subsets $\Delta_i\subseteq\Xi\subsetneq \Delta$ such that for each $1\leq i\leq m$
    \begin{enumerate}
    	\item[(a)] $(\Delta_i,\sigma|_{\Delta_i})$ and $(\Xi,\sigma|_{\Xi})$ are transitive subshifts of finite type.
    	\item[(b)] $\htop(\sigma,  \Delta_i)>h_{\mu_i}(\sigma)-\tilde{\eta}$ and consequently,   $\htop(\sigma,  \Xi)>\sup\{h_{\kappa}(\sigma):\kappa\in F\}-\tilde{\eta}.$
    	\item[(c)] $d_H(\tilde{F},  \mathcal{M}(\sigma,  \Xi))<\tilde{\zeta}$,   $d_H(\mu_i,  \mathcal{M}(\sigma,  \Delta_i))<\tilde{\zeta}$.
    	\item[(d)] There is a positive integer $L$ such that for any $z\in\Xi$ one has  $d(\sigma^{j+mL}(z),x)<\eps$ for some $0\leq j\leq L-1$ and any $m\in\Z$.
    \end{enumerate}
    Then by item (a)(d) we have item 1 and 4. By (\ref{equation-AA}) and item (c) we have $d_H(\mu_\rho^i,  \mathcal{M}(\F,  \Delta^i_\rho))<\zeta$ and thus  $F$ is contained in the $\zeta$-neighborhood of $\mathcal{M}(\F,  \Xi_\rho).$ Note that for any $\theta_i\in[0,  1]$ with $\sum_{i=1}^m\theta_i=1$ one has
    $$\mathcal{R}(\sum_{i=1}^{m}\theta_i\mu_i)=\sum_{i=1}^{m}\frac{\theta_i\int \rho d\mu_i}{\sum_{j=1}^{m}\theta_j\int\rho d\mu_j}\mathcal{R}(\mu_i).$$
    So $\mathcal{M}(\F,  \Xi_\rho)$ is contained in the $\zeta$-neighborhood of $F$ and thus we have item 3.
    
    By the variational principle of the topological entropy there is $\nu_i\in\mathcal{M}(\sigma,\Delta_i)$ such that $h_{\nu_i}(\sigma)>\htop(\sigma,  \Delta_i)-\tilde{\eta}>h_{\mu_i}(\sigma)-2\tilde{\eta}.$ Then
    \begin{equation*}
    	\htop(\Delta^i_\rho)\geq h_{\mathcal{R}(\nu_i)}(\F)=\frac{h_{\nu_i}(\sigma)}{\int \rho d\nu_i} >\frac{\int\rho {\rm d}\mu_i}{\int\rho {\rm d}\nu_i} \cdot \Big(h_{\mu^i_{\rho}}(\F) -\frac{2\tilde{\eta}}{\int\rho {\rm d}\mu_i}\Big) >h_{\mu^i_{\rho}}(\F) -\eta
    \end{equation*}
    by (\ref{equation-AA}). We obtain item 2.
\end{proof}

\subsection{Basic  sets and proof of Theorem \ref{Mainlemma-convex-by-horseshoe5}}\label{section-basic set}

Following the classical arguments of Bowen~\cite{Bowen72,Bowen73} on %the study of basic sets for 
Axiom A vector fields,  every basic set is semi-conjugate to a suspension flow over a transitive subshift of finite type with a continuous roof function.  Now we recall some basic results from \cite{Bowen72,Bowen73}. Readers can also see \cite{BS00b}.

Given $X\in\mathscr{X}^1(M)$, and a basic set $\Lambda.$  By \cite[Theorem 2.5]{Bowen73} there is a family of closed sets $R_{1}, \ldots, R_{k}$ and a  positive number $\alpha$ so that  the following properties hold:
\begin{description}
	\item[(1)] $\Lambda=\bigcup_{t \in[- \alpha,0]} \phi_{t}\left(\bigcup_{i=1}^{k} R_{i}\right).$
	\item[(2)] There exists a transitive subshift of finite type $(\Delta,\sigma)$ and a continuous and onto map $\pi:\Delta\to \bigcup_{i=1}^{k} R_{i}$ such that $\pi\circ\sigma=T\circ \pi$ ($T$ is the transfer map whose definition will be recalled later).
	\item[(3)] $\pi$ is one-to-one off $\bigcup_{n\in\mathbb{Z}} \sigma^n\left(\pi^{-1}(\bigcup_{i=1}^{k} \partial R_{i})\right)$.
\end{description}
By item (1) we define the transfer function $\tau: \Lambda \rightarrow[0, \infty)$ by
$$
\tau(x)=\min \left\{t>0: \phi_{t}(x) \in \bigcup_{i=1}^{k} R_{i}\right\}.
$$
There is a positive number $\beta$ such that $\tau(x)\geq \beta$ for any $x\in \Lambda.$ And there is a metric on $\Delta$ such that $\tau \circ \pi$ is Lipschitz. Let $T: \Lambda \rightarrow \bigcup_{i=1}^{k} R_{i}$ be the transfer map given by $T(x)=\phi_{\tau(x)}(x)$. The restriction of $T$ to $\bigcup_{i=1}^{k} R_{i}$ is invertible and $\bigcup_{i=1}^{k} R_{i}$ is a $T$-invariant set.
We can obtain a suspension flow $f_t:\Delta_\rho\to\Delta_\rho$ over the transitive subshift of finite type $(\Delta,\sigma)$ with Lipschitz roof function $\rho=\tau \circ \pi.$
Then we extend $\pi$ to a finite-to-one surjection $\pi: \Delta_{\rho} \rightarrow \Lambda$ by $\pi(x, s)=\left(\phi_{s} \circ \pi\right)(x)$ for every $(x, s) \in \Delta_{\rho}.$ We have
$$
\pi \circ f_{t}=\phi_{t} \circ \pi
$$
and $\pi$ is one-to-one off $\bigcup_{t\in\mathbb{R}} f_t\left(\pi^{-1}(\bigcup_{i=1}^{k} \partial R_{i})\right)$.

The boundary  of every $R_i$ consists of two part $\partial R_i=\partial^s R_i\cup \partial^u R_i.$ Denote 
$$\Delta^{s} \Lambda=\phi_{[0, \alpha]} (\bigcup_{i=1}^{k}\partial^s R_i)\text{ and }
\Delta^{u} \Lambda=\phi_{[-\alpha,0]} (\bigcup_{i=1}^{k}\partial^u R_i).$$
By \cite[Proposition 2.6]{Bowen73}, one has $\phi_t(\Delta^{s} \Lambda)\subset \Delta^{s} \Lambda$ and $\phi_{-t}(\Delta^{u} \Lambda)\subset \Delta^{u} \Lambda$ for any $t\geq 0.$ In fact, in the proof of \cite[Proposition 2.6]{Bowen73}, it's also shown that 
\begin{equation}\label{equation-AB}
	T(\bigcup_{i=1}^{k}\partial^s R_i)\subset \bigcup_{i=1}^{k}\partial^s R_i\text { and }T^{-1}(\bigcup_{i=1}^{k}\partial^u R_i)\subset \bigcup_{i=1}^{k}\partial^u R_i.
\end{equation}

Now we give the proof of Theorem \ref{Mainlemma-convex-by-horseshoe5}.

\noindent{\bf Proof of Theorem \ref{Mainlemma-convex-by-horseshoe5}}: 
	Let $f_t:\Delta_\rho\to\Delta_\rho$ be the associated suspension flow of the basic set $\Lambda$. Since $\pi^{-1}(\cup_{i=1}^{k}\partial^s R_i)$ is a proper closed subset of $\Delta,$ then there is $\tilde{x}\in \Delta$ and $\tilde{\eps}>0$ such that $$B(\tilde{x},\tilde{\eps})\cap\pi^{-1}( \bigcup_{i=1}^{k}\partial^s R_i)=\emptyset,$$ where $B(\tilde{x},\tilde{\eps})=\{y\in \Delta: d(\tilde{x},y)<\tilde{\eps}\}.$ 
	\begin{claim}\label{claim-1}
		If there is a point $z\in \Delta$ and  an increasing sequence $\{t_m\}_{m=-\infty}^{+\infty}$ so that $f_{t_m}((z,0))\in \Delta\times \{0\},$  $d_{\Delta_{\rho}}(f_{t_m}((z,0)),(\tilde{x},0))<\tilde{\eps}$ for any $m\in\Z,$ and $\lim\limits_{m\to+\infty}t_m=\lim\limits_{m\to-\infty}-t_m=+\infty,$ then $z\notin \bigcup_{n\in\mathbb{Z}} \sigma^n\left(\pi^{-1}(\bigcup_{i=1}^{k} \partial R_{i})\right).$
	\end{claim} 
    \begin{proof}
    	Without loss of generality, we assume  $\sigma^l(z)\in \pi^{-1}(\bigcup_{i=1}^{k} \partial ^s R_{i})$ for some  integer $l\in\mathbb{Z}.$ Then by (\ref{equation-AB}) we have $\sigma^{l+n}(z)\in \pi^{-1}(\bigcup_{i=1}^{k} \partial ^s R_{i})$ for any $n\geq 0.$ Note that $f_{t_m}((z,0))\in \Delta\times \{0\}$ implies there is  an integer $n_m$ such that $f_{t_m}((z,0))=(\sigma^{n_m}(z),0).$ So by $\lim\limits_{m\to+\infty}t_m=+\infty,$ there is an integer $n_m>l$ such that $d(\sigma^{n_m}(z),\tilde{x})=d_{\Delta_{\rho}}(f_{t_m}((z,0)),(\tilde{x},0))<\tilde{\eps}.$ This contradicts $B(\tilde{x},\tilde{\eps})\cap\pi^{-1}( \bigcup_{i=1}^{k}\partial^s R_i)=\emptyset.$
    \end{proof}

Fix $F=\cov\{\mu_i\}_{i=1}^m\subseteq \mathcal{M}(\Phi,\Lambda),$ and  $\eta,  \zeta>0.$ Then there is $\tilde{\mu}_i\in\mathcal{M}(\F,\Delta_{\rho})$ such that $\mu_i=\tilde{\mu}_i\circ \pi^{-1}.$
Since $\pi$ is continuous, there is $\tilde{\zeta}>0$ such that $d^{*}(\tilde{\mu}\circ\pi^{-1},\tilde{\nu}\circ\pi^{-1})<\zeta$ for any $\tilde{\mu},\tilde{\nu}\in \mathcal{M}(\F,\Delta_{\rho})$ with $d^*(\tilde{\mu},\tilde{\nu})<\tilde{\zeta}.$
For $\tilde{F}=\cov\{\tilde{\mu}_i\}_{i=1}^m\subseteq \mathcal{M}(\F,\Delta_{\rho}),$ $\tilde{x}$ and $\eta,\tilde{\zeta},\tilde{\eps}>0,$ by Proposition \ref{Mainlemma-convex-by-horseshoe4} there exist compact $\F$-invariant subsets $\Delta^i_\rho\subseteq\Xi_\rho\subsetneq \Delta_\rho$ such that for each $1\leq i\leq m$
\begin{enumerate}
	\item[(a)] $f_{t}|_{\Delta^i_\rho}: \Delta^i_\rho \rightarrow \Delta^i_\rho$  and $f_{t}|_{\Xi_\rho}: \Xi_\rho\rightarrow \Xi_\rho$ are suspension flows over transitive subshifts of finite type $(\Delta_i,\sigma|_{\Delta_i})$ and $(\Xi,\sigma|_{\Xi})$  with the roof function $\rho.$ 
	\item[(b)] $\htop(\Delta^i_\rho)>h_{\tilde{\mu}_i}(\F)-\eta$ and consequently,   $\htop(\Xi_\rho)>\sup\{h_{\kappa}(\F):\kappa\in F\}-\eta.$
	\item[(c)] $d_H(\tilde{F},  \mathcal{M}(\F,  \Xi_\rho))<\tilde{\zeta}$,   $d_H(\tilde{\mu}_i,  \mathcal{M}(\F,  \Delta^i_\rho))<\tilde{\zeta}$.
	\item[(d)] For any $z\in\Xi$ there is an increasing sequence $\{t_m\}_{m=-\infty}^{+\infty}$ such that $f_{t_m}((z,0))\in \Delta\times \{0\},$ $d_{\Delta_{\rho}}(f_{t_m}((z,0)),(\tilde{x},0))<\tilde{\eps}$ for any $m\in\Z,$ and $\lim\limits_{m\to+\infty}t_m=\lim\limits_{m\to-\infty}-t_m=+\infty$.
\end{enumerate}
By Claim \ref{claim-1} and item(d),  we have $$\Delta_i\cap \bigcup_{n\in\mathbb{Z}} \sigma^n\left(\pi^{-1}(\bigcup_{i=1}^{k} \partial R_{i})\right)=\emptyset\text{ and }\Xi\cap \bigcup_{n\in\mathbb{Z}} \sigma^n\left(\pi^{-1}(\bigcup_{i=1}^{k} \partial R_{i})\right)=\emptyset.$$
Then 
$$\Delta^i_\rho\cap \bigcup_{t\in\mathbb{R}} f_t\left(\pi^{-1}(\bigcup_{i=1}^{k} \partial R_{i})\right)=\emptyset\text{ and }\Xi_\rho\cap \bigcup_{t\in\mathbb{R}} f_t\left(\pi^{-1}(\bigcup_{i=1}^{k} \partial R_{i})\right)=\emptyset.$$
This implies $\pi$ is one-to-one on $\Delta^i_\rho$ and $\Xi_\rho.$ So $\Lambda_i=\pi(\Delta^i_\rho)$ and $\Theta=\pi(\Xi_\rho)$ are the horseshoes we want. \qed

\section{Asymptotically  additive families and almost additive families}\label{Almost Additive}
In this section we consider Theorem \ref{Thm:basic-set1} and \ref{Thm:basic-set2} in the more general context of asymptotically  additive families and almost additive families of continuous functions and give abstract conditions on which the results of Theorem \ref{Thm:basic-set1} and \ref{Thm:basic-set2} hold.
\subsection{Nonadditive thermodynamic formalism for flows}\label{subsection-equilbrium}
In this subsection we  recall  a few basic notions and results on the nonadditive thermodynamic formalism for flows.
Consider a  continuous flow $\Phi=(\phi_t)_{t\in\mathbb{R}}$ on a compact metric space $(M,d).$ Let $\Lambda\subset M$ be a compact $\phi_t$-invariant set.   
Given $x\in \Lambda$ and $t,\varepsilon>0$, we consider the set
$$
B_{t}(x, \varepsilon)=\left\{y \in \Lambda: d\left(\phi_{s}(y), \phi_{s}(x)\right)<\varepsilon \text { for any } s \in[0, t]\right\}.
$$
Moreover, let $a=\left(a_t\right)_{t\geq 0}$ be a family of continuous functions $a_t:\Lambda\to \mathbb{R}$ with {\it tempered variation}, that is, such that
$$
\lim _{\varepsilon \rightarrow 0} \limsup_{t \rightarrow \infty} \frac{\gamma_{t}(a, \varepsilon)}{t}=0,
$$
where
$$
\gamma_{t}(a, \varepsilon)=\sup \left\{\left|a_{t}(y)-a_{t}(x)\right|: y \in B_{t}(x, \varepsilon),\ x\in \Lambda \right\}.
$$
Given $\varepsilon>0$, we say that a set $\Gamma \subset \Lambda \times \mathbb{R}_{0}^{+}$covers $\Lambda$ if
$$
\bigcup_{(x, t) \in \Gamma} B_{t}(x, \varepsilon) \supset \Lambda,
$$
and we write
$$
a(x, t, \varepsilon)=\sup \left\{a_{t}(y): y \in B_{t}(x, \varepsilon)\right\} \quad \text { for }(x, t) \in \Gamma.
$$
For each  $\alpha \in \mathbb{R}$, let
\begin{equation}\label{equation-AC}
	M(\Lambda, a, \alpha, \varepsilon)=\lim _{T \rightarrow+\infty} \inf _{\Gamma} \sum_{(x, t) \in \Gamma} \exp (a(x, t, \varepsilon)-\alpha t)
\end{equation}
with the infimum taken over all countable sets $\Gamma \subset \Lambda \times[T,+\infty)$ covering $\Lambda$. When $\alpha$ goes from $-\infty$ to $+\infty$, the quantity in (\ref{equation-AC}) jumps from $+\infty$ to 0 at a unique value and so one can define
$$
P\left(a, \Lambda, \varepsilon\right)=\inf \{\alpha \in \mathbb{R}: M(\Lambda, a, \alpha, \varepsilon)=0\}.
$$
Moreover, the limit
$$
P\left(a, \Lambda\right)=\lim _{\varepsilon \rightarrow 0} P\left(a, \Lambda, \varepsilon\right).
$$
exists and is called the nonadditive topological pressure of the family $a$ on the set $\Lambda$. 

We recall that a family of functions $a=\left(a_t\right)_{t\geq 0}$ on $\Lambda$ is said to be {\it almost additive} with respect to the flow $(\phi_t)_{t\in\mathbb{R}}$ if there is a constant $C>0$ such that $$
-C+a_{t}+a_{s} \circ \phi_{t} \leqslant a_{t+s} \leqslant a_{t}+a_{s} \circ \phi_{t}+C
$$
for every $t, s\geq 0$. 
Let $A(\Phi,\Lambda)$ be the set of all almost additive families of continuous functions $a=\left(a_{t}\right)_{t \geqslant 0}$ on $\Lambda$ with tempered variation such that
$$
\sup _{t \in[0, s]}\left\|a_{t}\right\|_{\infty}<\infty \quad \text { for some } s>0.
$$
From \cite[Proposition 4]{BH21} the limit $\lim \limits_{t \rightarrow \infty} \frac{1}{t} \int_{\Lambda} a_{t} \mathrm{~d} \mu$ exists for any $a=\left(a_{t}\right)_{t \geq 0}\in A(\Phi,\Lambda)$ and any $\mu\in \mathcal{M}(\Phi,\Lambda),$ and 
the function,
\begin{equation}\label{equation-P}
	\mathcal{M}(\Phi,\Lambda) \ni \mu \mapsto \lim _{t \rightarrow \infty} \frac{1}{t} \int_{\Lambda} a_{t} \mathrm{~d} \mu,
\end{equation}
is continuous with the weak* topology.
Let $a=\left(a_{t}\right)_{t \geq 0}\in A(\Phi,\Lambda).$
Then by \cite[Theorem 2.1 and Lemma 2.4]{BH212}, we have 
\begin{equation}\label{equation-BB}
	\begin{split}
		P(a,\Lambda)&=\sup_{\mu\in\mathcal{M}(\Phi,\Lambda)}\left(h_{\mu}(\Phi)+\lim _{t \rightarrow \infty} \frac{1}{t} \int_{M} a_{t} d \mu\right)\\
		&\sup_{\mu\in\mathcal{M}_{erg}(\Phi,\Lambda)}\left(h_{\mu}(\Phi)+\lim _{t \rightarrow \infty} \frac{1}{t} \int_{M} a_{t} d \mu\right).
	\end{split}
\end{equation}
A measure $\mu\in \mathcal{M}(\Phi,\Lambda)$ is said to be an {\it equilibrium measure} for the almost additive family $a$ if $$P(a,\Lambda)=h_\mu(\Phi)+\lim\limits_{t \rightarrow \infty} \frac{1}{t} \int_{M} a_{t} d \mu.$$ We say that $a$ has {\it bounded variation} if  for every $\kappa>0$ there exists $\varepsilon>0$ such that
$$
\left|a_{t}(x)-a_{t}(y)\right|<\kappa \quad \text { whenever } y \in B_{t}(x, \varepsilon).
$$
Following the  arguments \cite[Section 3.2]{BH212} we have the following result on uniqueness of equilibrium measure for suspension flows.
\begin{lemma}\label{lemma-G}
	Suppose  $f_{t}: \Delta_{\rho} \rightarrow \Delta_{\rho}$ is a suspension flow over a transitive subshift of finite type $(\Delta,\sigma)$ with a Lipschitz roof function $\rho.$  Let $a=\left(a_{t}\right)_{t \geq 0}$ be an almost additive family of continuous functions on $\Delta_\rho$ with bounded variation such that $P(a,\Delta_\rho)=0$ and $\sup _{t \in[0, s]}\left\|a_{t}\right\|_{\infty}<\infty$ for some $s>0$. Then there is a unique equilibrium measure $\mu_{a}$ for $a$.
\end{lemma}
\begin{proof}
	Define a sequence of functions $c_n:\Delta\to\mathbb{R}$ by 
	$$c_n(x)=a_{\rho_n(x)}(x)$$
	where $\rho_n(x)=\sum_{i=0}^{n-1}\rho(\sigma^i(x)).$
	Following the  arguments \cite[Section 3.2]{BH212} we have
	\begin{description}
		\item[(1)] $c=(c_n)_{n\in\mathbb{N}}$ is an almost additive sequence of continuous functions, that is, there is a constant $\tilde{C}>0$ such that for every $n, m \in \mathbb{N}$ we have
		$$
		-\tilde{C}+c_{n}+c_{m} \circ \sigma^{n} \leqslant c_{n+m} \leqslant \tilde{C}+c_{n}+c_{m} \circ \sigma^{n}.
		$$
		\item[(2)] $c=(c_n)_{n\in\mathbb{N}}$ has bounded variation, that is, there exists $\varepsilon>0$ for which
		$
		\sup _{n \in \mathbb{N}} \gamma_{n}(c, \varepsilon)<\infty
		$
		with $$\gamma_{n}(c, \varepsilon)=\sup \left\{\left|c_{n}(x)-c_{n}(y)\right|: x,y\in\Delta,\ d\left(\sigma^{k} (x), \sigma^{k} (y)\right)<\varepsilon \text { for } k=0, \ldots, n\right\}.$$ 
		\item[(3)] For any ergodic measure $\nu\in\mathcal{M}(\sigma,\Delta)$ and $\mu=\mathcal{R}(\nu)$ (recall Subsection~\ref{subsec:suspension})
		$$
		h_{\mu}(\F)+\lim _{t \rightarrow \infty} \frac{1}{t} \int_{\Delta_\rho} a_{t} d \mu=\left(h_{\nu}\left(\sigma\right)+\lim _{n \rightarrow \infty} \frac{1}{n} \int_{\Delta} c_{n} d \nu\right) / \int_{\Delta} \rho d \nu
		$$
	\end{description}
	By item(3) we have $h_{\mu}(\F)+\lim \limits_{t \rightarrow \infty} \frac{1}{t} \int_{\Delta_\rho} a_{t} d \mu=0$ if and only if $h_{\nu}\left(\sigma\right)+\lim\limits_{n \rightarrow \infty} \frac{1}{n} \int_{\Delta} c_{n} d \nu=0$
	Since $P(a,\Delta_\rho)=0$, then  $\mu$ is an equilibrium measure for $a$ if and only if $\nu$ is an equilibrium measure for $c$ by (\ref{equation-BB}).
	
	It's proved in \cite[Theorem 5.3]{DHT} that for every homeomorphism which is transitive, expansive and has the shadowing property, and for every almost additive sequence of continuous functions with bounded variation, there is a unique equilibrium measure.
	Recall from \cite{Walters2} a subshift satisfies the shadowing property if and only if it is a subshift of finite type. As a subsystem of two-side full shift, it is expansive. 
	Then there is a unique equilibrium measure $\nu_{c}$ for $c$, and $\mathcal{R}(\mu_c)$ is the unique equilibrium measure for $a$.
\end{proof}
\begin{corollary}\label{corollary-C}
	Suppose  $f_{t}: \Delta_{\rho} \rightarrow \Delta_{\rho}$ is a suspension flow over a transitive subshift of finite type $(\Delta,\sigma)$ with a Lipschitz roof function $\rho.$  Let $a=\left(a_{t}\right)_{t \geq 0}$ be an almost additive family of continuous functions on $\Delta_\rho$ with bounded variation such that $\sup _{t \in[0, s]}\left\|a_{t}\right\|_{\infty}<\infty$ for some $s>0$. Then there is a unique equilibrium measure $\mu_{a}$ for $a$.
\end{corollary}
\begin{proof}
	For any $t \geq 0,$ define
	$$
	b_{t}=a_{t}-P(a) t
	$$
	Then $b$ is  an almost additive family with bounded variation and satisfies $\sup _{t \in[0, s]}\left\|b_{t}\right\|_{\infty}<\infty$ and $P(b)=0$. For each $\F$-invariant probability measure $\mu$ on $\Delta_{\rho}$ we have
	$$
	\frac{1}{t} \int_{X} a_{t} d \mu=\frac{1}{t} \int_{X} b_{t} d \mu-P(a) .
	$$
	This implies that $a$ and $b$ have the same equilibrium measures.  Then by Lemma \ref{lemma-G} we complete the proof.
\end{proof}

\begin{lemma}
	Given $X\in\mathscr{X}^1(M)$ and an invariant compact  set $\Lambda$. Assume that there exists a suspension flow $f_{t}: \Delta_{\rho} \rightarrow \Delta_{\rho}$ over a transitive subshift of finite type $(\Delta,\sigma)$ with a Lipschitz roof function $\rho$, and a finite-to-one continuous surjection $\pi: \Delta_{\rho} \rightarrow \Lambda$, such that $\pi\circ f_t=\phi_t\circ \pi$. Then for every almost additive family of continuous functions $a=\left(a_{t}\right)_{t \geq 0}$  on $\Lambda$ with bounded variation such that $\sup _{t \in[0, s]}\left\|a_{t}\right\|_{\infty}<\infty$ for some $s>0,$  there is a unique equilibrium measure $\mu_{a}$ for $a$.
\end{lemma}
\begin{proof}
	Define $\tilde{a}_t=a_t\circ \pi$ for any $t\geq 0.$ Since $\pi$ is continuous, $\tilde{a}=\left(\tilde{a}_{t}\right)_{t \geq 0}$ is an almost additive family of continuous functions on $\Delta_\rho$ with bounded variation and $\sup _{t \in[0, s]}\left\|\tilde{a}_{t}\right\|_{\infty}<\infty.$ By Corollary \ref{corollary-C}, there is a unique equilibrium measure $\mu_{\tilde{a}}$ for $\tilde{a}$. Since $\pi$ is finite-to one, we have $h_{\mu_{\tilde{a}}\circ\pi^{-1}}(\Phi)=h_{\mu_{\tilde{a}}}(\F)$ and thus
	$$h_{\mu_{\tilde{a}}\circ\pi^{-1}}(\Phi)+\lim\limits_{t\to\infty}\frac{1}{t}\int a_t d\mu_{\tilde{a}}\circ\pi^{-1}=h_{\mu_{\tilde{a}}}(\F)+\lim\limits_{t\to\infty}\frac{1}{t}\int \tilde{a}_t d\mu_{\tilde{a}}=P(\tilde{a})\geq P(a).$$
	Then by (\ref{equation-BB}), $\mu_{\tilde{a}}\circ\pi^{-1}$ is an equilibrium measure for $a.$ To show that the measure is unique, suppose $\nu$ is an equilibrium measure for $a,$ and choose $\mu\in\mathcal{M}(\F,\Delta_{\rho})$ with $\mu\circ\pi^{-1}=\nu.$ Then $$h_{\mu}(\F)+\lim\limits_{t\to\infty}\frac{1}{t}\int \tilde{a}_t d\mu=h_{\nu}(\Phi)+\lim\limits_{t\to\infty}\frac{1}{t}\int a_t d\nu= P(a)=P(\tilde{a}).$$
	Thus $\mu=\mu_{\tilde{a}}$ by the uniqueness of $\mu_{\tilde{a}}.$ So  $\nu=\mu\circ\pi^{-1}=\mu_{\tilde{a}}\circ\pi^{-1}$ is the unique equilibrium measure for $a$.
\end{proof}

Then we have the following combining with Subsection \ref{section-basic set} and Definition \ref{definition-hyperbolic}. 
\begin{theorem}\label{theorem-AA}
	Given $X\in\mathscr{X}^1(M)$ and an invariant compact  set $\Lambda$. If $\Lambda$ is a basic set or a horseshoe of $\phi_{t},$ then for every almost additive family of continuous functions $a=\left(a_{t}\right)_{t \geq 0}$  on $\Lambda$ with bounded variation such that $\sup _{t \in[0, s]}\left\|a_{t}\right\|_{\infty}<\infty$ for some $s>0,$  there is a unique equilibrium measure $\mu_{a}$ for $a$.
\end{theorem}

\subsection{Conditional variational principles}
\subsubsection{Conditional variational principle of almost additive families}
Consider a  continuous flow $\Phi=(\phi_t)_{t\in\mathbb{R}}$ on a compact metric space $(M,d).$ Let $\Lambda\subset M$ be a compact $\phi_t$-invariant set.   
Let $d \in \mathbb{N}$ and  $(A, B) \in A(\Phi,\Lambda)^{d} \times A(\Phi,\Lambda)^{d}$ where 
$$
A=\left(a^{1}, \ldots, a^{d}\right) \text { and } B=\left(b^{1}, \ldots, b^{d}\right)
$$
and  $a^{i}=\left(a_{t}^{i}\right)_{t\in\R}$ and $b^{i}=\left(b_{t}^{i}\right)_{t\in\R}$. Assume that
\begin{equation}\label{equation-O}
	\liminf _{t \rightarrow \infty} \frac{b_{t}^{i}(x)}{t}>0 \quad \text { and } \quad b_{t}^{i}(x)>0
\end{equation}
for every $i=1, \ldots, d, x \in \Lambda$, and $t \in \mathbb{R}$. 
Given $\alpha=\left(\alpha_1, \ldots, \alpha_{d}\right) \in \mathbb{R}^{d}$ we define:
$$
R_{A,B}(\alpha)=\bigcap_{i=1}^{d}\left\{x \in \Lambda: \lim _{t \rightarrow \infty} \frac{a_{t}^{i}(x)}{b_{t}^{i}(x)}=\alpha_{i}\right\}.
$$
We also define the map $\mathcal{P}_{A,B}: \mathcal{M}(\Phi,\Lambda) \rightarrow \mathbb{R}$ by:
\begin{equation}\label{equation-A}
	\mathcal{P}_{A,B}(\mu) 
	=\lim _{t \rightarrow \infty}\left(\frac{\int a_{t}^{1} d \mu}{\int b_{t}^{1} d \mu}, \ldots, \frac{\int a_{t}^{d} d \mu}{\int b_{t}^{d} d \mu}\right).
\end{equation}
\eqref{equation-P} ensures  that the function $\mathcal{P}_{A,B}$ is continuous.
Denote
$$L_{A,B}=\{\mathcal{P}_{A,B}(\mu):\mu\in \mathcal{M}(\Phi,\Lambda)\}.$$

Let $E(\Phi,\Lambda) \subset A(\Phi,\Lambda)$ be the set of  families with a unique equilibrium measure. 
Define the sequence of constant functions $u=(u_t)_{t\geq 0}$ with $u_t\equiv t$ for any $t\geq 0.$
In \cite{BH21} L. Barreira and C. Holanda give the conditional variational principle as following.
\begin{theorem}\cite[Theorem 9]{BH21}\label{BarreiraDoutor2009-theorem3}
	Suppose that $\Phi=(\phi_t)_{t\in\mathbb{R}}$  is a continuous flow on a compact metric space $(M,d),$ and $\Lambda\subset M$ is a compact invariant set such  that  the entropy function $\mathcal{M}(\Phi,\Lambda) \ni \mu \mapsto h_\mu(\Phi)$ is upper semi-continuous.   Let $d \in \mathbb{N}$ and $(A, B) \in A(\Phi,\Lambda)^{d} \times A(\Phi,\Lambda)^{d}$ such that $B$ satisfies \eqref{equation-O} and
	$
	\text{span}\left\{a^1, b^1,\cdots,a^d,b^d,u\right\} \subseteq E(\Phi,\Lambda).
	$ 
	If $\alpha \in \mathrm{Int}(L_{A,B})$, then $R_{A,B}(\alpha)\neq \emptyset$, and the following properties hold:
	\begin{description}
		\item[(1)] $\htop(R_{A,B}(\alpha))$ satisfies the variational principle:
		$$
		\htop(R_{A,B}(\alpha))=\max \left\{h_{\mu}(\Phi): \mu \in \mathcal{M}(\Phi,\Lambda) \text { and } \mathcal{P}_{A,B}(\mu)=\alpha\right\}.
		$$
		\item[(2)] There is an ergodic measure $\mu_{\alpha} \in \mathcal{M}(\Phi,\Lambda) $ with $\mathcal{P}_{A,B}\left(\mu_{\alpha}\right)=\alpha, \mu_{\alpha}\left(R_{A,B}(\alpha)\right)=1$, and
		$$
		\htop(R_{A,B}(\alpha))=h_{\mu_{\alpha}}(\Phi).
		$$
	\end{description}
\end{theorem}
\begin{remark}
	In fact,  L. Barreira and C. Holanda in \cite[Theorem 9]{BH21} give the proof of Theorem \ref{BarreiraDoutor2009-theorem3} under the assumption $d=1.$ However, L. Barreira and P. Doutor in \cite[Theorem 3]{BarreiraDoutor2009} proved the case of discrete time for any $d\geq 1.$ Combing the two proofs, one can obtain Theorem \ref{BarreiraDoutor2009-theorem3} for any $d\geq 1.$
\end{remark}

\subsubsection{Conditional variational principle of asymptotically  additive families}
Let $\Phi=\left(\phi_t\right)_{t \in \mathbb{R}}$ be a continuous flow on a compact metric space $(M,d),$ and $\Lambda\subset M$ be a compact $\phi_t$-invariant set. A family of functions $a=\left(a_t\right)_{t \geq 0}$ is said to be asymptotically additive with respect to $\Phi$ on $\Lambda$ if for each $\varepsilon>0$ there exists a function $b_{\varepsilon}: \Lambda \rightarrow \mathbb{R}$ such that
$$
\limsup _{t \rightarrow \infty} \frac{1}{t}\sup_{x\in\Lambda}|a_t(x)-\int_0^t\left(b_{\varepsilon} \circ \phi_s(x)\right) d s| \leq \varepsilon
$$
Let $AA(\Phi,\Lambda)$ be the set of all asymptotically additive families of continuous functions $a=\left(a_{t}\right)_{t \geqslant 0}$ on $\Lambda$ with tempered variation such that
$$
\sup _{t \in[0, s]}\left\|a_{t}\right\|_{\infty}<\infty \quad \text { for some } s>0.
$$
Proceeding as in \cite{FH2010}, one can see that every almost additive family of functions is asymptotically additive.

Now assume that $\Phi$ is   expansive on $\Lambda.$ C. Holanda in \cite[Corollary 6]{Holanda} proved that for any $a=\left(a_t\right)_{t \geq 0}\in AA(\Phi,\Lambda),$ there exists a continuous function $b: \Lambda \rightarrow \mathbb{R}$ such that
$$
\lim \limits_{t \rightarrow \infty} \frac{1}{t} \int_{\Lambda} a_{t} \mathrm{~d} \mu=\int_{\Lambda} bd\mu
$$
for  any $\mu\in \mathcal{M}(\Phi,\Lambda),$ 
This implies that  
the function,
\begin{equation}\label{equation-UU}
	\mathcal{M}(\Phi,\Lambda) \ni \mu \mapsto \lim _{t \rightarrow \infty} \frac{1}{t} \int_{\Lambda} a_{t} \mathrm{~d} \mu,
\end{equation}
is continuous with the weak* topology.
Let $d \in \mathbb{N}$ and  $(A, B) \in AA(\Phi,\Lambda)^{d} \times AA(\Phi,\Lambda)^{d}$ where 
$$
A=\left(a^{1}, \ldots, a^{d}\right) \text { and } B=\left(b^{1}, \ldots, b^{d}\right)
$$
and  $a^{i}=\left(a_{t}^{i}\right)_{t\in\R}$ and $b^{i}=\left(b_{t}^{i}\right)_{t\in\R}$. Assume that
\begin{equation}\label{equation-UV}
	\begin{split}
		&\lim _{t \rightarrow \infty} \frac{1}{t} \int_{\Lambda} b_t^i d \mu \geq 0 \quad \text { for all } \mu \in \mathcal{M}(\Phi,\Lambda)\\
		 &\text{ with equality only permitted when } \lim _{t \rightarrow \infty} \frac{1}{t} \int_{\Lambda} a_t^i d \mu \neq 0
	\end{split}
\end{equation}
for every $i=1, \ldots, d.$ 
Given $\alpha=\left(\alpha_1, \ldots, \alpha_{d}\right) \in \mathbb{R}^{d}$ we define:
$$
R_{A,B}(\alpha)=\bigcap_{i=1}^{d}\left\{x \in \Lambda: \lim _{t \rightarrow \infty} \frac{a_{t}^{i}(x)}{b_{t}^{i}(x)}=\alpha_{i}\right\}.
$$
We also define the map $\mathcal{P}_{A,B}: \mathcal{M}(\Phi,\Lambda) \rightarrow \mathbb{R}$ by:
\begin{equation}\label{equation-UW}
	\mathcal{P}_{A,B}(\mu) 
	=\lim _{t \rightarrow \infty}\left(\frac{\int a_{t}^{1} d \mu}{\int b_{t}^{1} d \mu}, \ldots, \frac{\int a_{t}^{d} d \mu}{\int b_{t}^{d} d \mu}\right).
\end{equation}
\eqref{equation-UU} ensures that  the function $\mathcal{P}_{A,B}$ is continuous. Denote
$$L_{A,B}=\{\mathcal{P}_{A,B}(\mu):\mu\in \mathcal{M}(\Phi,\Lambda)\}.$$

Without using the uniqueness of equilibrium assumption, C. Holanda  obtain a conditional variational principle for asymptotically additive families of continuous functions.
\begin{theorem}\cite[Theorem 11]{Holanda}\label{condition-principle}
	Given $X\in\mathscr{X}^1(M)$ and an invariant compact  set $\Lambda$. Assume that $\Lambda$ is a basic set or a horseshoe of $\phi_{t}.$  Let $d \in \mathbb{N}$ and $(A, B) \in AA(\Phi,\Lambda)^{d} \times AA(\Phi,\Lambda)^{d}$ satisfying \eqref{equation-UV}.
	If $\alpha \in \mathrm{Int}(L_{A,B})$, then $R_{A,B}(\alpha)\neq \emptyset$, and the following properties hold:
	\begin{description}
		\item[(1)] $\htop(R_{A,B}(\alpha))$ satisfies the variational principle:
		$$
		\htop(R_{A,B}(\alpha))=\sup \left\{h_{\mu}(\Phi): \mu \in \mathcal{M}(\Phi,\Lambda) \text { and } \mathcal{P}_{A,B}(\mu)=\alpha\right\}.
		$$
		\item[(2)] For any $\varepsilon>0$ there exists an ergodic measure $\mu_{\alpha} \in \mathcal{M}(\Phi,\Lambda) $ such that $\mathcal{P}_{A,B}\left(\mu_{\alpha}\right)=\alpha, \mu_{\alpha}\left(R_{A,B}(\alpha)\right)=1$, and
		$$
		|\htop(R_{A,B}(\alpha))-h_{\mu_{\alpha}}(\Phi)|<\varepsilon.
		$$
	\end{description}
\end{theorem}
\begin{remark}
		Using the work of Climenhaga  \cite[Theorem 3.3]{Climen2013} and Cuneo \cite[Theorem 1.2]{Cuneo2020}, and the conclusion that every Hölder  continuous function has a unique equilibrium measure with respect to the map $T$(recall that $T$ is the transfer map, see Section \ref{section-basic set}), C. Holanda in \cite[Theorem 11]{Holanda} give the proof of Theorem \ref{condition-principle} under the assumption $d=1$ and $\Lambda$ is a mixing basic set. However, \cite[Theorem 3.3]{Climen2013} is stated for any $d\geq 1$ and by \cite[Example 2]{Bowen74} every Hölder  continuous function has a unique equilibrium measure with respect to the map $T$ when $\Lambda$ is a just transitive basic set. So one can obtain Theorem \ref{condition-principle} for any $d\geq 1$ and any basic set $\Lambda.$ When $\Lambda$ is a horseshoe of $\phi_{t},$ we can verify the results also hold following the argument of \cite[Theorem 11]{Holanda}.
\end{remark}

\subsection{Abstract conditions on which  Theorem \ref{Thm:basic-set1} and \ref{Thm:basic-set2} hold}\label{section-abstract conditions}
Let $\Phi=\left(\phi_t\right)_{t \in \mathbb{R}}$ be a continuous flow on a compact metric space $(M,d),$ and $\Lambda\subset M$ be a compact $\phi_t$-invariant set. Assume that $\Phi$ is   expansive on $\Lambda.$ 
Let $d \in \mathbb{N}$ and  $(A, B) \in AA(\Phi,\Lambda)^{d} \times AA(\Phi,\Lambda)^{d}.$
For any $\alpha\in  L_{A,B},$ denote  $$M_{A,B}(\alpha)=\{\mu\in \mathcal{M}(\Phi,\Lambda):\mathcal{P}_{A,B}(\mu)=\alpha\},\ M_{A,B}^{erg}(\alpha)=\{\mu\in \mathcal{M}_{erg}(\Phi,\Lambda):\mathcal{P}_{A,B}(\mu)=\alpha\}.$$
Then $M_{A,B}(\alpha)$ is closed in $\mathcal{M}(\Phi,\Lambda)$ since the map $\mathcal{P}_{A,B}$ is continuous.

Let $\chi:\mathcal{M}(\Phi,\Lambda)\to \mathbb{R}$ be a continuous function. 
We define the pressure of $\chi$ with respect to $\mu$ by $P(\Phi,\chi,\mu)=h_\mu(\Phi)+\chi(\mu).$ 
 Now we give  a result in the  context of asymptotically additive families.

\begin{theorem}\label{thm-Almost-Additive}
	Suppose that $\Phi=(\phi_t)_{t\in\mathbb{R}}$  is a continuous flow on a compact metric space $(M,d),$ and $\Lambda\subset M$ is a compact invariant set such  that $\Phi$ is expansive on $\Lambda.$  Let $d \in \mathbb{N}$ and $(A, B) \in AA(\Phi,\Lambda)^{d} \times AA(\Phi,\Lambda)^{d}$ 
	satisfying \eqref{equation-UV}. Let $\chi:\mathcal{M}(\Phi,\Lambda)\to \mathbb{R}$ be a continuous function.  Assume that the following holds: 
	for any $F=\cov\{\mu_i\}_{i=1}^m\subseteq \mathcal{M}(\Phi,\Lambda),$ and any $\eta,  \zeta>0$,   there are compact invariant subsets $\Lambda_i\subseteq\Theta\subsetneq \Lambda$ such that for each $i\in\{1,2,\cdots,m\}$
	\begin{description}
		\item[(1)] for any $a\in \mathrm{Int}(\mathcal{P}_{A,B}(\mathcal{M}(\Phi,\Theta)))$ and any $\varepsilon>0,$ there exists an ergodic measure $\mu_a$ supported on $\Theta$ with $\mathcal{P}_{A,B}(\mu_a)=a$
		such that $|h_{\mu_a}(\Phi)-H(\Phi,a,\Theta)|<\varepsilon,$ where $H(\Phi,a,\Theta)=\sup\{h_\mu(\Phi):\mu\in \mathcal{M}(\Phi,\Theta) \text{ and }\mathcal{P}_{A,B}(\mu)=a\}.$
		\item[(2)] $\htop(\Lambda_i)>h_{\mu_i}(\Phi)-\eta.$
		\item[(3)] $d_H(K,  \mathcal{M}(\Phi,  \Theta))<\zeta$,   $d_H(\mu_i,  \mathcal{M}(\Phi,  \Lambda_i))<\zeta.$
	\end{description}
	Then for any $\alpha\in \mathrm{Int}(L_{A,B}),$ any $\mu_0\in M_{A,B}(\alpha)$ and any $\eta,  \zeta>0$, there is $\nu\in M_{A,B}^{erg}(\alpha)$ such that $d^*(\nu,\mu_0)<\zeta$ and $|P(\Phi,\chi,\nu)-P(\Phi,\chi,\mu_0)|<\eta.$
\end{theorem}

For any $\alpha\in  L_{A,B},$ denote $$H_{A,B}(\Phi,\chi,\alpha)=\sup\{P(\phi,\chi,\mu):\mu\in M_{A,B}(\alpha)\}.$$ In particular, when $\chi\equiv0,$ we write
$$H_{A,B}(\Phi,\alpha)=H_{A,B}(\Phi,0,\alpha)=\sup\{h_\mu(\Phi):\mu\in M_{A,B}(\alpha)\}.$$
We list two conditions for $\chi:$
\begin{description}
	\item[(A.1)] For any $\mu_1, \mu_2 \in \mathcal{M}(\Phi,\Lambda)$ with $\chi(\mu_1) \neq \chi(\mu_2)$
	\begin{equation}\label{equation-W}
		\beta(\theta):=\chi(\theta \mu_1+(1-\theta) \mu_2)\text{ is strictly monotonic on }[0,1].
	\end{equation} 
	\item[(A.2)] For any $\mu_1, \mu_2 \in \mathcal{M}(\Phi,\Lambda)$ with $\chi(\mu_1) = \chi(\mu_2)$
	\begin{equation}\label{equation-AF}
		\beta(\theta):=\chi(\theta \mu_1+(1-\theta) \mu_2)\text{ is constant on }[0,1].
	\end{equation} 
\end{description}

Now we give  abstract conditions on which  Theorem \ref{Thm:basic-set1} and \ref{Thm:basic-set2} hold in the  context of asymptotically additive families.
\begin{theorem}\label{thm-Almost-Additive2}
	Suppose that $\Phi=(\phi_t)_{t\in\mathbb{R}}$  is a continuous flow on a compact metric space $(M,d),$ and $\Lambda\subset M$ is a compact invariant set such  that $\Phi$ is expansive on $\Lambda.$    Let $d \in \mathbb{N}$ and $(A, B) \in AA(\Phi,\Lambda)^{d} \times AA(\Phi,\Lambda)^{d}$ satisfying \eqref{equation-UV}. Let $\chi:\mathcal{M}(\Phi,\Lambda)\to \mathbb{R}$ be a continuous function satisfying \eqref{equation-W} and \eqref{equation-AF}. 
	Assume that for any $\alpha\in \mathrm{Int}(L_{A,B}),$ any $\mu_0\in M_{A,B}(\alpha)$ and any $\eta,  \zeta>0$, there is $\nu\in M_{A,B}^{erg}(\alpha)$ such that $d^*(\nu,\mu_0)<\zeta$ and $|P(\Phi,\chi,\nu)-P(\Phi,\chi,\mu_0)|<\eta.$
	If $\{\mu\in \mathcal{M}(\Phi,\Lambda):h_{\mu}(\Phi)=0\}$ is dense in $\mathcal{M}(\Phi,\Lambda),$ then we have
	\begin{description}
		\item[(1)] 
		For any $\alpha\in \mathrm{Int}(L_{A,B}),$ any $\mu_0\in M_{A,B}(\alpha),$ any $\max_{\mu\in M_{A,B}(\alpha)}\chi(\mu)\leq c\leq P(\phi,\chi,\mu_0)$ and any $\eta,  \zeta>0$, there is $\nu\in M_{A,B}^{erg}(\alpha)$ such that $d^*(\nu,\mu_0)<\zeta$ and $|P(\Phi,\chi,\nu)-c|<\eta.$ 
		\item[(2)]  For any $\alpha\in \mathrm{Int}(L_{A,B})$ and $\max_{\mu\in M_{A,B}(\alpha)}\chi(\mu)\leq c< H_{A,B}(\Phi,\chi,\alpha),$ the set $\{\mu\in M_{A,B}^{erg}(\alpha):P(\Phi,\chi,\mu)=c\}$ is residual in $\{\mu\in M_{A,B}(\alpha):P(\Phi,\chi,\mu)\geq c\}.$ If further there is an invariant measure $\tilde{\mu}$ with $S_{\tilde{\mu}}=\Lambda$, then for any $\alpha\in \mathrm{Int}(L_{A,B})$ and $\max_{\mu\in M_{A,B}(\alpha)}\chi(\mu)\leq c< H_{A,B}(\Phi,\chi,\alpha),$ the set $\{\mu\in M_{A,B}^{erg}(\alpha):P(\Phi,\chi,\mu)=c,\ S_\mu=\Lambda\}$ is residual in $\{\mu\in M_{A,B}(\alpha):P(\Phi,\chi,\mu)\geq c\}.$
		\item[(3)] The set $\{(\mathcal{P}_{A,B}(\mu), P(\Phi,\chi,\mu)):\mu\in \mathcal{M}(\Phi,\Lambda),\ \alpha=\mathcal{P}_{A,B}(\mu)\in\mathrm{Int}(L_{A,B}),\ \max_{\mu\in M_{A,B}(\alpha)}\chi(\mu)\\
		\leq P(\Phi,\chi,\mu)< H_{A,B}(\Phi,\chi,\alpha)\}$ coincides with $\{(\mathcal{P}_{A,B}(\mu), P(\Phi,\chi,\mu)):\mu\in \mathcal{M}_{erg}(\Phi,\Lambda),\ \alpha\\
		=\mathcal{P}_{A,B}(\mu)\in\mathrm{Int}(L_{A,B}), \max_{\mu\in M_{A,B}(\alpha)}\chi(\mu)\leq P(\Phi,\chi,\mu)< H_{A,B}(\Phi,\chi,\alpha)\}.$
	\end{description}
\end{theorem}
\begin{example}\label{example-1}
	The function $\chi:\mathcal{M}(\Phi,\Lambda)\to \mathbb{R}$ can be defined as  following:
	\begin{description}
		\item[(1)] $\chi\equiv0.$ Then $P(\Phi,\chi,\mu)=h_\mu(\Phi)$ is the metric entropy of $\mu.$
		\item[(2)] $\chi(\mu)=\int g d \mu$ with a continuous function $g.$ Then from the weak{*}-topology on $\mathcal{M}(\Phi,\Lambda),$ $\chi:\mathcal{M}(\Phi,\Lambda)\to \mathbb{R}$ is a continuous function. $P(\Phi,\chi,\mu)=h_\mu(\Phi)+\chi(\mu)$ is the pressure of $g$ with respect to $\mu.$
		\item[(3)] $\chi(\mu)=\lim\limits_{t \rightarrow \infty} \frac{1}{t} \int a_{t} d \mu$ with an asymptotically  additive families of continuous functions $a=\left(a_{t}\right)_{t\in\R}$ on $\Lambda.$ Then $\chi:\mathcal{M}(\Phi,\Lambda)\to \mathbb{R}$ is a continuous function from (\ref{equation-UU}). 
		$P(\Phi,\chi,\mu)=h_\mu(\Phi)+\lim\limits_{t \rightarrow \infty} \frac{1}{t} \int a_{t} d \mu$ is the pressure of $a$ with respect to $\mu.$ 
	\end{description}
    Furthermore, if $\chi$ is defined as above, then   \eqref{equation-W} and \eqref{equation-AF} hold for $\chi$  since it is affine.  
\end{example}
\begin{remark}\label{remark-AB}
	In Theorem \ref{thm-Almost-Additive} and \ref{thm-Almost-Additive2}, the  expansivity of $\Phi$ is used to guarantee the function $\label{equation-UU}
		\mathcal{M}(\Phi,\Lambda) \ni \mu \mapsto \lim _{t \rightarrow \infty} \frac{1}{t} \int_{\Lambda} a_{t} \mathrm{~d} \mu$ is continuous   and the metric entropy $\mathcal{M}(\Phi,\Lambda) \ni \mu \mapsto h_\mu(\Phi)$ is upper semi-continuous. When we consider almost additive families, by (\ref{equation-P}) we only need to assume that the metric entropy $\mathcal{M}(\Phi,\Lambda) \ni \mu \mapsto h_\mu(\Phi)$ is upper semi-continuous.
\end{remark}

\subsection{Proof of Theorem \ref{thm-Almost-Additive}}
\subsubsection{Some lemmas}
We  establish several auxiliary results.
For any $r\in \mathbb{R},$ denote $r^+=\{s\in\mathbb{R}:s>r\}$ and $r^-=\{s\in\mathbb{R}:s<r\}.$  For any $d \in \mathbb{N},$ $r=\left(r_1, \ldots, r_{d}\right)\in\mathbb{R}^d$ and $\xi=\left(\xi_1, \ldots, \xi_{d}\right)\in\{+,-\}^d$, we define $$r^\xi=\{s=\left(s_1, \ldots, s_{d}\right)\in\mathbb{R}^d:s_i\in r_i^{\xi_i} \text{ for }i=1,2,\cdots,d\}.$$ We denote $F^d=\{\left(\frac{p_1}{q_1}, \ldots, \frac{p_d}{q_d}\right):p_i,q_i\in\mathbb{R}\text{ and }q_i>0 \text{ for any }1\leq i\leq d\}.$
It is easy to check that
\begin{lemma}\label{lemma-E}
	Let $b_i=\frac{p^i}{q^i}\in F^1$ for $i=1,2.$
	\begin{description}
		\item[(1)] If $b_1=b_2,$ then $\frac{\theta p^1+(1-\theta)p^2}{\theta q^1+(1-\theta)q^2}=b_1=b_2$ for any $\theta\in[0,1].$
		\item[(2)] If $b_1\neq b_2,$ then $\frac{\theta p^1+(1-\theta)p^2}{\theta q^1+(1-\theta)q^2}$ is strictly monotonic on $\theta\in[0,1].$
	\end{description}
\end{lemma}
We can obtain the following result using mathematical induction.
\begin{lemma}\label{lemma-C}
	Let $d \in \mathbb{N}$ and $a=\left(\frac{p_1}{q_1}, \ldots, \frac{p_d}{q_d}\right) \in F^{d}.$  If $\{b_\xi=\left(\frac{p_1^\xi}{q_1^\xi}, \ldots, \frac{p_d^\xi}{q_d^\xi}\right)\}_{\xi\in\{+,-\}^d}\subseteq F^d$ are $2^d$ numbers satisfies $b_\xi\in a^\xi$ for any $\xi\in \{+,-\}^d,$ then there are $2^d$ numbers $\{\theta_\xi\}_{\xi\in\{+,-\}^d}\subseteq [0,1]$ such that $\sum_{\xi\in\{+,-\}^d}\theta_\xi=1$ and $$ \frac{\sum_{\xi\in\{+,-\}^d}\theta_\xi p^{\xi}_i}{\sum_{\xi\in\{+,-\}^d}\theta_\xi q^{\xi}_i}=\frac{p_i}{q_i}\text{ for any }1\leq i\leq d.$$
\end{lemma}

From Lemma \ref{lemma-C} we have
\begin{corollary}\label{corollary-A}
	Suppose that $\Phi=(\phi_t)_{t\in\mathbb{R}}$  is a continuous flow on a compact metric space $(M,d),$ and $\Lambda\subset M$ is a compact invariant set such  that $\Phi$ is expansive on $\Lambda.$  Let $d \in \mathbb{N}$ and $(A, B) \in AA(\Phi,\Lambda)^{d} \times AA(\Phi,\Lambda)^{d}$ 
	satisfying \eqref{equation-UV}. Then for any $\alpha\in \mathrm{Int}(L_{A,B}),$ and $2^d$ invariant measures $\{\mu_\xi\}_{\xi\in\{+,-\}^d}$ with 
	$$\mathcal{P}_{A,B}\left(\mu_\xi\right)\in \alpha^\xi \text{ for any }\xi\in \{+,-\}^d,$$ there are $2^d$ numbers $\{\theta_\xi\}_{\xi\in\{+,-\}^d}\subseteq [0,1]$ such that $\sum_{\xi\in\{+,-\}^d}\theta_\xi=1$ and $\mathcal{P}_{A,B}\left(\sum_{\xi\in\{+,-\}^d}\theta_\xi\mu_\xi\right)\\ = \alpha.$
\end{corollary}

\begin{lemma}\label{lemma-D}
	Suppose that $\Phi=(\phi_t)_{t\in\mathbb{R}}$  is a continuous flow on a compact metric space $(M,d),$ and $\Lambda\subset M$ is a compact invariant set such  that $\Phi$ is expansive on $\Lambda.$  Let $d \in \mathbb{N}$ and $(A, B) \in AA(\Phi,\Lambda)^{d} \times AA(\Phi,\Lambda)^{d}$ 
	satisfying \eqref{equation-UV}.  Then for any $\alpha\in \mathrm{Int}(L_{A,B}),$ any $\mu\in M_{A,B}(\alpha)$ and any $\eta,\zeta>0$, there are $2^d$ invariant measures $\{\mu_\xi\}_{\xi\in\{+,-\}^d}$ such that { for any }$\xi\in \{+,-\}^d$
	$$\mathcal{P}_{A,B}\left(\mu_\xi\right)\in \alpha^\xi, h_{\mu_\xi}(\Phi)>h_{\mu}(f)-\eta \text{ and } d^*(\mu_\xi,\mu)<\zeta.$$
\end{lemma}
\begin{proof}
	By $\alpha\in \mathrm{Int}(L_{A,B})$ there is $\nu_\xi\in \mathcal{M}(\Phi,\Lambda)$ such that $\mathcal{P}_{A,B}\left(\nu_\xi\right)\in a^\xi$ for any $\xi\in \{+,-\}^d.$ We choose $\tau_\xi\in(0,1)$ close to $1$ such that $\mu_\xi=\tau_\xi\mu+(1-\tau_\xi)\nu_\xi$ satisfies 
	$$h_{\mu_\xi}(\Phi)>h_{\mu}(\Phi)-\eta \text{ and } d^*(\mu_\xi,\mu)<\zeta \text{ for any }\xi\in \{+,-\}^d.$$ Then we have $\mathcal{P}_{A,B}\left(\mu_\xi\right)\in a^\xi$ by $\tau^\xi>0$ and Lemma \ref{lemma-E}(2).
\end{proof}

\subsubsection{Proof of Theorem \ref{thm-Almost-Additive}}
Fix $\alpha\in \mathrm{Int}(L_{A,B}),$ $\mu_0\in M_{A,B}(\alpha)$ and $\eta,  \zeta>0.$ 
Since $\Phi$ is expansive on $\Lambda,$  the metric entropy $\mathcal{M}(\Phi,\Lambda) \ni \mu \mapsto h_\mu(\Phi)$ is upper semi-continuous.   Hence there is $0<\zeta'<\zeta$ such that  for any $\omega\in \mathcal{M}(\Phi,\Lambda)$ with $d^*(\mu_0,\omega)<\zeta'$ we have
\begin{equation}\label{equation-C}
	h_{\omega}(\Phi)<h_{\mu_0}(\Phi)+\frac{3\eta}{8} \text{ and }|\chi(\omega)-\chi(\mu_0)|<\frac{\eta}{2}.
\end{equation} 
By Lemma \ref{lemma-D} there are $2^d$ invariant measures $\{\mu_\xi\}_{\xi\in\{+,-\}^d}$ such that \text{ for any }$\xi\in \{+,-\}^d$
\begin{equation}\label{equation-U}
	\mathcal{P}_{A,B}\left(\mu_\xi\right)\in \alpha^\xi, h_{\mu_\xi}(\Phi)>h_{\mu_0}(\Phi)-\frac{\eta}{8} \text{ and } d^*(\mu_\xi,\mu_0)<\frac{\zeta'}{2}.
\end{equation}
Since the map $\mathcal{P}_{A,B}$ is continuous,  there is $0<\zeta''<\zeta'$ such that for any $\omega_\xi \in \mathcal{M}(\Phi,\Lambda)$ with $d^*(\omega_\xi,\mu_\xi)<\zeta''$ one has
\begin{equation}\label{equation-B}
	\mathcal{P}_{A,B}\left(\omega_\xi\right)\in \alpha^\xi.
\end{equation} 
For the $2^d$ invariant measures $\{\mu_\xi\}_{\xi\in\{+,-\}^d},$ there are compact invariant subsets $\Lambda_\xi\subseteq\Theta\subsetneq \Lambda$ such that for each $\xi\in \{+,-\}^d$
\begin{description}
	\item[(1)] for any $a\in \mathrm{Int}(\mathcal{P}_{A,B}(\mathcal{M}(\Phi,\Theta)))$ and any $\varepsilon>0,$ there exists an ergodic measure $\mu_a$ supported on $\Theta$ with $\mathcal{P}_{A,B}(\mu_a)=a$
	such that $|h_{\mu_a}(\Phi)-H(\Phi,a,\Theta)|<\varepsilon.$ 
	\item[(2)] $\htop( \Lambda_\xi)>h_{\mu_\xi}(\Phi)-\frac{\eta}{8}.$
	\item[(3)] $d_H(\cov\{\mu_\xi\}_{\xi\in \{+,-\}^d},  \mathcal{M}(\Phi,\Theta))<\frac{\zeta''}{2}$,   $d_H(\mu_\xi,  \mathcal{M}(\Phi,\Lambda_\xi))<\frac{\zeta''}{2}.$
\end{description}
By item(2) and the variational principle, there is $\nu_\xi\in \mathcal{M}(\Phi,\Lambda_\xi)$ such that $$h_{\nu_\xi}(\Phi)>\htop( \Lambda_\xi)-\frac{\eta}{8}>h_{\mu_\xi}(\Phi)-\frac{2\eta}{8}>h_{\mu_0}(\Phi)-\frac{3\eta}{8}.$$
Then by item(3) and \eqref{equation-B}, we have $\mathcal{P}_{A,B}\left(\nu_\xi\right)\in \alpha^\xi.$ 
By Corollary \ref{corollary-A} there is $2^d$ numbers $\{\theta_\xi\}_{\xi\in\{+,-\}^d}\subseteq [0,1]$ such that $\sum_{\xi\in\{+,-\}^d}\theta_\xi=1$ and $\mathcal{P}_{A,B}\left(\nu'\right)= \alpha$ where $\nu'=\sum_{\xi\in\{+,-\}^d}\theta_\xi\nu_\xi.$ 
Then on one hand we have 
\begin{equation}\label{equation-E}
		H(\Phi,a,\Theta)
		\geq h_{\nu'}(\Phi)\geq \min\{h_{\nu_\xi}(\Phi):\xi\in \{+,-\}^d\}
		> h_{\mu_0}(\Phi)-\frac{3\eta}{8}.
\end{equation}
On the other hand, by item(3),  \eqref{equation-U} and \eqref{equation-C} we have 
\begin{equation}\label{equation-F}
	H(\Phi,a,\Theta)<h_{\mu_0}(\Phi)+\frac{3\eta}{8}.
\end{equation}
Now by item(1) there exists an ergodic measure $\nu$ supported on $\Theta$ with $\mathcal{P}_{A,B}(\mu_a)=a$
such that $|h_{\nu}(\Phi)-H(\Phi,a,\Theta)|<\frac{\eta}{8}.$ 
Then $\nu\in M_{A,B}^{erg}(\alpha)$ and by \eqref{equation-E}, \eqref{equation-F} we have $|h_{\nu}(\Phi)-h_{\mu_0}(\Phi)|<\frac{\eta}{2}.$ By item(3) and  \eqref{equation-U} we have $d^*(\nu,\mu_0)<\zeta'<\zeta.$ Finally, by \eqref{equation-C} we have $|\chi(\nu)-\chi(\mu_0)|<\frac{\eta}{2},$ and thus $|P(\Phi,\chi,\nu)-P(\Phi,\chi,\mu_0)|<\eta.$ \qed

\subsection{Proof of Theorem \ref{thm-Almost-Additive2}}\label{section-proof}

\subsubsection{Some lemmas}
\begin{lemma}\label{lemma-B}
	Suppose that $\Phi=(\phi_t)_{t\in\mathbb{R}}$  is a continuous flow on a compact metric space $(M,d),$ and $\Lambda\subset M$ is a compact invariant set. Let $V$ be a convex subset of $\mathcal{M}(\Phi,\Lambda).$ If there is an invariant measure $\mu_V\in V$ with $S_{\mu_V}=\Lambda,$ then $\{\mu\in V:S_\mu=\Lambda\}$ is residual in $V.$ 
\end{lemma}
\begin{proof}
	Since $\{\mu\in \mathcal{M}(\Phi,\Lambda):S_\mu=\Lambda\}$ is either empty or a dense $G_\delta$ subset of   $\mathcal{M}(\Phi,\Lambda)$ from \cite[Proposition 21.11]{DGS}. So if there is an invariant measure $\mu_V\in V$ with $S_{\mu_V}=\Lambda,$ then $\{\mu\in \mathcal{M}(\Phi,\Lambda):S_\mu=\Lambda\}$ is a dense $G_\delta$ subset of $\mathcal{M}(\Phi,M).$ Thus $\{\mu\in V:S_\mu=\Lambda\}$ is a $G_\delta$ subset of  $V.$ In addition, for any $\nu\in V$ and $\theta\in(0,1),$ we have $\nu_\theta=\theta\nu+(1-\theta)\mu_V\in V$ and $S_{\nu_\theta}=\Lambda.$ So $\{\mu\in V:S_\mu=\Lambda\}$ is dense in  $V,$ and thus is residual in $V.$ 
\end{proof}
\begin{lemma}\label{lemma-A}
	Suppose that $\Phi=(\phi_t)_{t\in\mathbb{R}}$  is a continuous flow on a compact metric space $(M,d),$ and $\Lambda\subset M$ is a compact invariant set such  that $\Phi$ is expansive on $\Lambda.$  Let $d \in \mathbb{N}$ and $(A, B) \in AA(\Phi,\Lambda)^{d} \times AA(\Phi,\Lambda)^{d}$ 
	satisfying \eqref{equation-UV}.   If  $\chi:\mathcal{M}(\Phi,\Lambda)\to \mathbb{R}$ is a continuous function  satisfies \eqref{equation-W} and \eqref{equation-AF}, then for any $\alpha\in \mathrm{Int}(L_{A,B})$ and any $\max_{\mu\in M_{A,B}(\alpha)}\chi(\mu)\leq c< H_{A,B}(\Phi,\chi,\alpha),$ the following properties hold:
	\begin{description}
		\item[(1)] If $\{\mu\in M_{A,B}^{erg}(\alpha):P(\Phi,\chi,\mu)\geq c\}$ is dense in $\{\mu\in M_{A,B}(\alpha):P(\Phi,\chi,\mu)\geq c\},$ then  $\{\mu\in M_{A,B}^{erg}(\alpha):P(\Phi,\chi,\mu)\geq c\}$ is residual in $\{\mu\in M_{A,B}(\alpha):P(\Phi,\chi,\mu)\geq c\}.$
		\item[(2)] If there is an invariant measure $\tilde{\mu}$ with $S_{\tilde{\mu}}=\Lambda$, then  $\{\mu\in M_{A,B}(\alpha):P(\Phi,\chi,\mu)\geq c,\ S_\mu=\Lambda\}$ is residual in $\{\mu\in M_{A,B}(\alpha):P(\Phi,\chi,\mu)\geq c\}.$   
		\item[(3)] If $\{\mu\in \mathcal{M}(\Phi,\Lambda):h_{\mu}(\Phi)=0\}$ is dense in $\mathcal{M}(\Phi,\Lambda),$ then   $\{\mu\in M_{A,B}(\alpha):P(\Phi,\chi,\mu)= c\}$ is residual in $\{\mu\in M_{A,B}(\alpha):P(\Phi,\chi,\mu)\geq c\}.$  
	\end{description}
\end{lemma}
\begin{proof}
	(1) From \cite[Proposition 5.7]{DGS}, $\mathcal{M}_{erg}(\Phi,\Lambda)$ is a $G_\delta$ subset of $\mathcal{M}(\Phi,\Lambda).$ Then $\{\mu\in M_{A,B}^{erg}(\alpha):P(\Phi,\chi,\mu)\geq c\}$ is a $G_\delta$ subset of $\{\mu\in M_{A,B}(\alpha):P(\Phi,\chi,\mu)\geq c\}.$ If $\{\mu\in M_{A,B}^{erg}(\alpha):P(\Phi,\chi,\mu)\geq c\}$ is dense in $\{\mu\in M_{A,B}(\alpha):P(\Phi,\chi,\mu)\geq c\},$ then  $\{\mu\in M_{A,B}^{erg}(\alpha):P(\Phi,\chi,\mu)\geq c\}$ is residual in $\{\mu\in M_{A,B}(\alpha):P(\Phi,\chi,\mu)\geq c\}.$
	
	(2) Since $\{\mu\in \mathcal{M}(\Phi,\Lambda):S_\mu=\Lambda\}$ is either empty or a dense $G_\delta$ subset of   $\mathcal{M}(\Phi,\Lambda)$ from \cite[Proposition 21.11]{DGS}. So if there is an invariant measure $\tilde{\mu}$ with $S_{\tilde{\mu}}=\Lambda$, then $\{\mu\in \mathcal{M}(\Phi,\Lambda):S_\mu=\Lambda\}$ is a dense $G_\delta$ subset of $\mathcal{M}(\Phi,\Lambda).$  
	
	Now we show that $\{\mu\in M_{A,B}(\alpha):P(\Phi,\chi,\mu)\geq c,\ S_\mu=\Lambda\}$ is nonempty. 
	By Lemma \ref{lemma-D} there exist $2^d$ invariant measures $\{\mu_\xi\}_{\xi\in\{+,-\}^d}$ such that 
	$$\mathcal{P}_{A,B}\left(\mu_\xi\right)\in a^\xi \text{ for any }\xi\in \{+,-\}^d.$$
	Since $\{\mu\in \mathcal{M}(\Phi,\Lambda):S_\mu=\Lambda\}$ is dense in $ \mathcal{M}(\Phi,\Lambda)$ and the map $\mathcal{P}_{A,B}$ is continuous, then there exists $\omega_\xi\in  \mathcal{M}(\Phi,\Lambda)$ close to $\mu_\xi$ such that 
	\begin{equation*}
		\mathcal{P}_{A,B}\left(\omega_\xi\right)\in a^\xi,\ S_{\omega_\xi}=\Lambda\text{ for any }\xi\in \{+,-\}^d.
	\end{equation*}
	By Corollary \ref{corollary-A} there is $2^d$ numbers $\{\theta_\xi\}_{\xi\in\{+,-\}^d}\subseteq [0,1]$ such that $\sum_{\xi\in\{+,-\}^d}\theta_\xi=1$ and $\mathcal{P}_{A,B}\left(\omega\right)= \alpha$ where $\omega=\sum_{\xi\in\{+,-\}^d}\theta_\xi\mu_\omega.$ Then $S_{\omega}=\Lambda.$ Since $c< H_{A,B}(\Phi,\chi,a),$ there is $\nu\in M_{A,B}(\alpha)$ such that $P(\Phi,\chi,\nu)>c.$
	By \eqref{equation-W} we can choose $\theta\in(0,1)$ close to $1$ such that $\mu'=\theta\nu+(1-\theta)\omega$ satisfies $P(\Phi,\chi,\mu')>c.$ 
	Then $\mu'\in \{\mu\in M_{A,B}(\alpha):P(\Phi,\chi,\mu)\geq c,\ S_\mu=\Lambda\}.$ Note that $\{\mu\in M_{A,B}(\alpha):P(\Phi,\chi,\mu)\geq c\}$ is a convex set by \eqref{equation-W}, \eqref{equation-AF} and Lemma  \ref{lemma-E}. So by Lemma \ref{lemma-B} we complete the proof of item(2).
	
	(3) Fix $\mu_0\in M_{A,B}(\alpha)$ with $P(\Phi,\chi,\mu_0)\geq c$  and $\zeta>0.$ By Lemma \ref{lemma-D} there exist $2^d$ invariant measures $\{\mu_\xi\}_{\xi\in\{+,-\}^d}$ such that 
	\begin{equation}\label{equation-X}
		\mathcal{P}_{A,B}\left(\mu_\xi\right)\in a^\xi \text{ and } d^*(\mu_\xi,\mu_0)<\frac{\zeta}{2} \text{ for any }\xi\in \{+,-\}^d.
	\end{equation}
	Since $\{\mu\in \mathcal{M}(\Phi,\Lambda):h_{\mu}(\Phi)=0\}$ is dense in $\mathcal{M}(\Phi,\Lambda)$ and  the function $\mathcal{P}_{A,B}$ is continuous,  then there exist $\nu_\xi\in \mathcal{M}(\Phi,\Lambda)$ close to $\mu_\xi$ such that 
	\begin{equation}\label{equation-Y}
		\mathcal{P}_{A,B}\left(\nu_\xi\right)\in a^\xi,\ h_{\nu_\xi}(\Phi)=0 \text{ and } d^*(\nu_\xi,\mu_\xi)<\frac{\zeta}{2} \text{ for each } \xi\in \{+,-\}^d.
	\end{equation}
	By Corollary \ref{corollary-A} there are $2^d$ numbers $\{\theta_\xi\}_{\xi\in\{+,-\}^d}\subseteq [0,1]$ such that $\sum_{\xi\in\{+,-\}^d}\theta_\xi=1$ and $\mathcal{P}_{A,B}\left(\nu'\right)= \alpha$ where $\nu'=\sum_{\xi\in\{+,-\}^d}\theta_\xi\nu_\xi.$ 
	Then by \eqref{equation-Y} $h_{\nu'}(\Phi)=0.$ By \eqref{equation-X} and \eqref{equation-Y} we have $d^*(\nu',\mu_0)<\zeta.$ 
	Now by \eqref{equation-W} we choose $\theta\in[0,1]$ such that  $\nu=\theta\mu_0+(1-\theta)\nu'$ satisfies $P(\Phi,\chi,\nu)=c.$ Then by Lemma \ref{lemma-E}(1)
	\begin{equation}\label{equation-Z}
		\mathcal{P}_{A,B}(\nu)=\alpha \text{ and } d^*(\nu,\mu_0)<\zeta.
	\end{equation}
	So $\{\mu\in M_{A,B}(\alpha):P(\Phi,\chi,\mu)= c\}$ is dense in $\{\mu\in M_{A,B}(\alpha):P(\Phi,\chi,\mu)\geq c\}.$  Since $\Phi$ is expansive on $\Lambda,$  the metric entropy $\mathcal{M}(\Phi,\Lambda) \ni \mu \mapsto h_\mu(\Phi)$ is upper semi-continuous.   Hence $\{\mu\in \mathcal{M}(\Phi,\Lambda):P(\Phi,\chi,\mu)\in[c,c+\frac{1}{n})\}$ is open in $\{\mu\in \mathcal{M}(\Phi,\Lambda):P(\Phi,\chi,\mu)\geq c\}$ for any $n\in\mathbb{N^{+}}.$  Then 
	\begin{equation*}
		\{\mu\in M_{A,B}(a):P(\Phi,\chi,\mu)\in[c,c+\frac{1}{n})\} \text{ is open and dense in }\{\mu\in M_{A,B}(a):P(\Phi,\chi,\mu)\geq c\},
	\end{equation*} 
	and thus 
	$\{\mu\in M_{A,B}(a):P(\Phi,\chi,\mu)=c\}$ is residual in $\{\mu\in M_{A,B}(a):P(\Phi,\chi,\mu)\geq c\}.$
\end{proof}

\subsubsection{Proof of Theorem \ref{thm-Almost-Additive2}}
\qquad (1) Fix $\alpha\in \mathrm{Int}(L_{A,B}),$ $\mu_0\in M_{A,B}(\alpha),$ $\max_{\mu\in M_{A,B}(\alpha)}\alpha(\mu)\leq c\leq P(\Phi,\chi,\mu_0)$  and $\eta,  \zeta>0.$ By Lemma \ref{lemma-A}(3), there exists $\nu'\in M_{A,B}(\alpha)$ such that $P(\Phi,\chi,\nu')=c$ and $d^*(\nu',\mu_0)<\frac{\zeta}{2}.$
For the $\alpha\in \mathrm{Int}(L_{A,B}),$ $\nu'\in M_{A,B}(\alpha)$ and $\eta,  \frac{\zeta}{2}>0,$ there is $\nu\in M_{A,B}^{erg}(\alpha)$ such that $d^*(\nu,\nu')<\frac{\zeta}{2}$ and $|P(\Phi,\chi,\nu)-P(\Phi,\chi,\nu')|<\eta.$ Then we complete the proof of item(1).

(2) Fix $\alpha\in \mathrm{Int}(L_{A,B})$ and $\max_{\mu\in M_{A,B}(\alpha)}\chi(\mu)\leq c<H_{A,B}(\Phi,\chi,\alpha).$ First we show that
\begin{equation}\label{equation-K}
	\{\mu\in M_{A,B}^{erg}(\alpha):P(f,\alpha,\mu)\geq c\} \text{ is dense in }\{\mu\in M_{A,B}(\alpha):P(f,\alpha,\mu)\geq c\}.
\end{equation} 
Let $\mu_0\in M_{A,B}(\alpha)$ be an invariant measure with $P(\Phi,\chi,\mu_0) \geq c$ and $\zeta>0$. If $P(\Phi,\chi,\mu_0)>c$, then there is  $\eta>0$ such that $c<c+\eta<P(\Phi,\chi,\mu_0).$ For the $\alpha\in \mathrm{Int}(L_{A,B}),$ $\mu_0\in M_{A,B}(\alpha)$ and $\eta,  \zeta>0,$ there exists an ergodic measure $\nu\in M_{A,B}^{erg}(\alpha)$ such that $d^*(\nu,\mu_0)<\zeta$ and $|P(\Phi,\chi,\nu)-P(\Phi,\chi,\mu_0)|<\eta.$ 
If $P(\Phi,\chi,\mu_0)=c$, then we can pick an invariant measure $\mu'\in M_{A,B}(\alpha)$ such that $c<P(\Phi,\chi,\mu') \leq  H_{A,B}(\Phi,\chi,a)$, and next pick a sufficiently small number $\theta \in(0,1)$ such that $d^*\left(\mu_0, \mu''\right)<\zeta / 2$, where $\mu''=(1-\theta) \mu_0+\theta\mu'.$ By \eqref{equation-W} we have $P(\Phi,\chi,\mu'')>c.$ By the same argument, there exists an ergodic measure $\nu\in M_{A,B}^{erg}(\alpha)$ such that $d^*(\nu,\mu'')<\zeta/2$ and $P(\Phi,\chi,\nu)> c.$  So $d^*(\nu,\mu_0)<\zeta.$

By \eqref{equation-K} and Lemma \ref{lemma-A}(1),
\begin{equation}\label{equation-M}
	\{\mu\in M_{A,B}^{erg}(\alpha):P(\Phi,\chi,\mu)\geq c\}\text{ is residual in }\{\mu\in M_{A,B}(\alpha):P(\Phi,\chi,\mu)\geq c\}.
\end{equation} 
By Lemma \ref{lemma-A}(3)
\begin{equation}\label{equation-L}
	\{\mu\in M_{A,B}(\alpha):P(\Phi,\chi,\mu)=c\} \text{ is residual in }\{\mu\in M_{A,B}(\alpha):P(\Phi,\chi,\mu)\geq c\}.
\end{equation}
If there is an invariant measure $\tilde{\mu}$ with $S_{\tilde{\mu}}=\Lambda$, then by Lemma \ref{lemma-A}(2) we have
\begin{equation}\label{equation-N}
	\{\mu\in M_{A,B}(\alpha):P(\Phi,\chi,\mu)\geq c,\ S_\mu=\Lambda\}\text{ is residual in }\{\mu\in M_{A,B}(\alpha):P(\Phi,\chi,\mu)\geq c\}
\end{equation}    
So by  \eqref{equation-M}, \eqref{equation-L} and \eqref{equation-N}, we complete the proof of item(2). 

(3) Fix $\alpha\in \mathrm{Int}(L_{A,B})$ and $\mu_0\in M_{A,B}(\alpha)$ with $\max\limits_{\mu\in M_{A,B}(\alpha)}\chi(\mu)\leq P(\Phi,\chi,\mu_0)< H_{A,B}(\Phi,\chi,\alpha).$ 
Then by item(2) the set $\{\mu\in M_{A,B}^{erg}(\alpha):P(\Phi,\chi,\mu)=P(\Phi,\chi,\mu_0)\}$ is residual in $\{\mu\in M_{A,B}(\alpha):P(\Phi,\chi,\mu)\geq P(\Phi,\chi,\mu_0)\}.$ In particular, there is $\mu_\alpha\in M_{A,B}^{erg}(\alpha)$ so that $P(\Phi,\chi,\mu_\alpha)=P(\Phi,\chi,\mu_0).$ \qed

\section{Proof of Theorem  \ref{Thm:basic-set1} and \ref{Thm:basic-set2}}\label{section-thm}
Now  we use ’multi-horseshoe’ dense property and the results of asymptotically  additive families  obtained in Section  \ref{Section:ED-HS-approximation} and \ref{Almost Additive} to give a more general result than Theorem  \ref{Thm:basic-set1} and \ref{Thm:basic-set2}. 
\begin{theorem}\label{thm-almost2}
	Let $X\in\mathscr{X}^1(M)$ and $\Lambda$ be a basic set.  Let $\chi:\mathcal{M}(\Phi,\Lambda)\to \mathbb{R}$ be a continuous function satisfying \eqref{equation-W} and \eqref{equation-AF}.  Let $d \in \mathbb{N}$ and $(A, B) \in AA(\Phi,\Lambda)^{d} \times AA(\Phi,\Lambda)^{d}$ satisfying \eqref{equation-UV}.
	Then:
	\begin{description}
		\item[(I)] For any $\alpha\in \mathrm{Int}(L_{A,B}),$ any $\mu_0\in M_{A,B}(\alpha),$ any $\max_{\mu\in M_{A,B}(\alpha)}\chi(\mu)\leq c\leq P(\phi,\chi,\mu_0)$ and any $\eta,  \zeta>0$, there is $\nu\in M_{A,B}^{erg}(\alpha)$ such that $d^*(\nu,\mu_0)<\zeta$ and $|P(\Phi,\chi,\nu)-c|<\eta.$ 
		\item[(II)] For any $\alpha\in \mathrm{Int}(L_{A,B})$ and $\max_{\mu\in M_{A,B}(\alpha)}\chi(\mu)\leq c< H_{A,B}(\Phi,\chi,\alpha),$ the set $\{\mu\in M_{A,B}^{erg}(\alpha):P(\Phi,\chi,\mu)=c,\ S_\mu=\Lambda\}$ is residual in $\{\mu\in M_{A,B}(\alpha):P(\Phi,\chi,\mu)\geq c\}.$
		\item[(III)] $\{(\mathcal{P}_{A,B}(\mu), P(\Phi,\chi,\mu)):\mu\in \mathcal{M}(\Phi,\Lambda),\ \alpha=\mathcal{P}_{A,B}(\mu)\in\mathrm{Int}(L_{A,B}),\ \max_{\mu\in M_{A,B}(\alpha)}\chi(\mu)\leq P(\Phi,\chi,\mu)< H_{A,B}(\Phi,\chi,\alpha)\}$ coincides with $\{(\mathcal{P}_{A,B}(\mu), P(\Phi,\chi,\mu)):\mu\in \mathcal{M}_{erg}(\Phi,\Lambda),\ \alpha=\mathcal{P}_{A,B}(\mu)\in\mathrm{Int}(L_{A,B}), \max_{\mu\in M_{A,B}(\alpha)}\chi(\mu)\leq P(\Phi,\chi,\mu)< H_{A,B}(\Phi,\chi,\alpha)\}.$  
	\end{description}
    If further $(A, B) \in A(\Phi,\Lambda)^{d} \times A(\Phi,\Lambda)^{d}$ such that $B$ satisfies \eqref{equation-O}, $a^i, b^i$  has bounded variation and  $\sup _{t \in[0, s]}\left\|a^i_{t}\right\|_{\infty}<\infty,$ $\sup _{t \in[0, s]}\left\|b^i_{t}\right\|_{\infty}<\infty$ for some $s>0$ and for any $1\leq i\leq d.$ Then:
    \begin{description}
    	\item[(IV)] The following two set are equal $$\{(\mathcal{P}_{A,B}(\mu), h_{\mu}(\Phi)):\mu\in \mathcal{M}(\Phi,\Lambda),\ \mathcal{P}_{A,B}(\mu)\in\mathrm{Int}(L_{A,B})\},$$ $$\{(\mathcal{P}_{A,B}(\mu), h_\mu(\Phi)):\mu\in \mathcal{M}_{erg}(\Phi,\Lambda),\ \mathcal{P}_{A,B}(\mu)\in\mathrm{Int}(L_{A,B})\}.$$
    \end{description} 
\end{theorem}
\begin{proof}
	It's known that each basic set is expansive, there is an invariant measure with full support, the set of periodic measures supported on in $\Lambda$ is a dense subset of  $\mathcal{M}(\Phi,\Lambda)$ and thus
	$\{\mu\in \mathcal{M}(\Phi,M):h_{\mu}(\Phi)=0\}$ is dense in $\mathcal{M}(\Phi,M)$.
	Then we obtain item(I)-(III) by Theorem \ref{Mainlemma-convex-by-horseshoe5}, Theorem \ref{condition-principle}, Theorem \ref{thm-Almost-Additive} and Theorem \ref{thm-Almost-Additive2}.
	
	Let $\chi=0,$ then by item(III) we have $\{(\mathcal{P}_{A,B}(\mu), h_{\mu}(\Phi)):\mu\in \mathcal{M}(\Phi,\Lambda),\ \alpha=\mathcal{P}_{A,B}(\mu)\in\mathrm{Int}(L_{A,B}),\ 0\leq h_{\mu}(\Phi)< H_{A,B}(\Phi,\alpha)\}$ coincides with $\{(\mathcal{P}_{A,B}(\mu), h_{\mu}(\Phi)):\mu\in \mathcal{M}_{erg}(\Phi,\Lambda),\ \alpha=\mathcal{P}_{A,B}(\mu)\in\mathrm{Int}(L_{A,B}),\ 0\leq h_{\mu}(\Phi)< H_{A,B}(\Phi,\alpha)\}.$  Combining with Theorem \ref{theorem-AA} and \ref{BarreiraDoutor2009-theorem3} we obtain item(IV).
\end{proof}

Now we give the proof of Theorem  \ref{Thm:basic-set1} and \ref{Thm:basic-set2}.

For a continuous function $g$, let $g_{t}=\int_{0}^{t} g(\phi_\tau(x))d\tau$ then $\sup _{t \in[0, 1]}\left\|g_{t}\right\|_{\infty}<\infty,$ $(g_t)_{t\geq 0}$ is a additive family of contiunous functions. So
let $\chi\equiv 0$ and $d=1$ in Theorem \ref{thm-almost2}(II),  we obtain Theorem \ref{Thm:basic-set1}.

Let $\chi\equiv 0$ and $d=2$ in Theorem \ref{thm-almost2}(II), for any $\alpha\in \mathrm{Int}(L_{g,h})$ and $$0\leq c< \max\left\{h_{\nu}(X):\nu\in M_{g,h}(\alpha)\right\},$$ the set $\{\mu\in M_{g,h}^{erg}(\alpha):h_\mu(X)=c,\ S_\mu=\Lambda\}$ is residual in $\{\mu\in M_{g,h}(\alpha):h_\mu(X)\geq c\}.$ Take $c=0,$ then the set $\{\mu\in M_{g,h}^{erg}(\alpha):S_\mu=\Lambda\}$ is dense in $M_{g,h}(\alpha).$ From \cite[Proposition 5.7]{DGS}, $\mathcal{M}_{erg}(\Phi,\Lambda)$ is a $G_\delta$ subset of $\mathcal{M}(\Phi,\Lambda).$ Then $M_{g,h}^{erg}(\alpha)$ is a $G_\delta$ subset of $M_{g,h}(\alpha),$ and thus   $M_{g,h}^{erg}(\alpha)$ is residual in $M_{g,h}(\alpha).$ Since $M_{g,h}(\alpha)$ is convex, by Lemma \ref{lemma-B} the set $\{\mu\in M_{g,h}(\alpha):S_\mu=\Lambda\}$ is residual in $M_{g,h}(\alpha).$ So $\{\mu\in M_{g,h}^{erg}(\alpha):S_\mu=\Lambda\}$ is residual in $M_{g,h}(\alpha)$ and we obtain Theorem \ref{Thm:basic-set2}.

\section{Singular hyperbolic attractors}\label{section-singular hyperbolic}
In this section, we consider singular hyperbolic attractors and give corresponding results on Question \ref{Conjecture-2}.

\subsection{Singular hyperbolicity and geometric Lorenz attractors}
The concept of singular hyperbolicity was introduced by Morales-Pacifico-Pujals~\cite{mpp} to describe the geometric structure of Lorenz attractors and these ideas were extended to higher dimensional cases in ~\cite{lgw,Me-Mo}. Let us recall this notion.

\begin{definition}\label{Def:singular-hyp}
	Given  a vector field $X\in\mathscr{X}^1(M)$, a compact and invariant set $\Lambda$ is {\it singular hyperbolic} if it admits a continuous ${\rm D}\phi_t$-invariant splitting $T_{\Lambda}M=E^{ss}\oplus E^{cu}$ and constants $C,\eta>0$ such that, 
	for any $x\in\Lambda$ and any $t\geq 0$, 
	\begin{itemize}
		\item $E^{ss}\oplus E^{cu}$ is a {\it dominated splitting} : $\|{\rm D}\phi_t|_{E^{ss}(x)}\|\cdot \|{\rm D}\phi_{-t}|_{E^{cu}(\phi_t(x))}\|< Ce^{-\eta t}$, and
		\item $E^{ss}$ is uniformly contracted by ${\rm D}\phi_t$ : $\|{\rm D}\phi_t(v)\|< Ce^{-\eta t}\|v\|$ for any $v\in E^{ss}(x)\setminus\{0\}$;
		\item $E^{cu}$ is {\it sectionally expanded} by ${\rm D}\phi_t$ :  $|\det{\rm D}\phi_t|_{V_x}| > Ce^{\eta t}$ for any 2-dimensional  subspace $V_x\subset E^{cu}_x$.
	\end{itemize}
	%a compact and invariant set $\Lambda$ is {\it negatively singular hyperbolic} if it is positively singular hyperbolic with respect to $-X$.  $\Lambda$ is singular hyperbolic if it is either positively or negatively singular hyperbolic.
\end{definition}
\begin{lemma}\cite[Theorem A and Lemma 2.9]{PYY}\label{lemma-semicontinuous}
	Given  a vector field $X\in\mathscr{X}^1(M)$ and an invariant compact  set $\Lambda$.  If $\Lambda$ is sectional hyperbolic, all the singularities in $\Lambda$  are hyperbolic, then $\Phi$ is entropy expansive on $\Lambda$ and thus the metric entropy function $\mathcal{M}(\Phi,\Lambda) \ni \mu \mapsto h_\mu(\Phi)$ is upper semi-continuous.
\end{lemma}

We recall the concept of homoclinic class of a hyperbolic periodic orbit. 

\begin{definition}\label{Def:homoclinic-class}
	Given a vector field $X\in\mathscr{X}^1(M)$, an  invariant compact subset $\Lambda\subset M$ is a {\it homoclinic class} if there exists a hyperbolic periodic point $p\in \Lambda \cap \text{Per}(X)$ so that 
	$$\Lambda\colon=\overline{W^{s}(\orb(p)) \pitchfork W^{u}(\orb(p))},$$ 
	that is, it is the closure of the points of transversal intersection between  stable and  unstable manifolds of the periodic
	orbit $\orb(p)$ of $p$.  
	We say a homoclinic class is \emph{non-trivial} if it is not reduced to a single hyperbolic periodic orbit.
\end{definition}

From \cite[Theorem 2.17]{AP} any nonempty homoclinic class $\Lambda$  contains a dense set of periodic orbits. Then there is an invariant measure $\tilde{\mu}$ with $S_{\tilde{\mu}}=\Lambda$ by \cite[Proposition 21.12]{DGS}.
\begin{lemma}\label{lemma-full-support}
	Let $X\in\mathscr{X}^1(M)$ and $\Lambda$ be a homoclinic class of $X$.  Then there is an invariant measure $\tilde{\mu}$ with $S_{\tilde{\mu}}=\Lambda.$
\end{lemma}

Now we give the definition of geometric Lorenz attractors following Guckenheimer and Williams~\cite{guck,gw,williams} for vector fields on a closed 3-manifold $M^3$.

\begin{definition}\label{Def:lorenz}
	We say $X\in\mathscr{X}^r(M^3)$ ($r\geq1$) exhibits a \emph{geometric Lorenz attractor} $\Lambda$, if $X$ has an attracting region $U\subset M^3$ such that  $\Lambda=\bigcap_{t>0}\phi^X_t(U)$ is a singular hyperbolic attractor and satisfies:
	\begin{itemize}
		\item There exists a unique singularity $\sigma\in\Lambda$ with three exponents $\lambda_1<\lambda_2<0<\lambda_3$, which satisfy $\lambda_1+\lambda_3<0$ and $\lambda_2+\lambda_3>0$.  
		\item $\Lambda$ admits a $C^r$-smooth cross section $\Sigma$ which is $C^1$-diffeomorphic to $[-1,1]\times[-1,1]$, such that $l=\{0\}\times[-1,1]=W^s_{\it loc}(\sigma)\cap\Sigma$, and
		for every $z\in U\setminus W^s_{\it loc}(\sigma)$, there exists $t>0$ such that $\phi_t^X(z)\in\Sigma$.
		\item Up to the previous identification, the Poincar\'e map $P:\Sigma\setminus l\rightarrow\Sigma$ is a skew-product map
		$$
		P(x,y)=\big( f(x)~,~H(x,y) \big), \qquad \forall(x,y)\in[-1,1]^2\setminus l.
		$$
		Moreover, it satisfies
		\begin{itemize}
			\item $H(x,y)<0$ for $x>0$, and $H(x,y)>0$ for $x<0$; \color{black} 
			\item $\sup_{(x,y)\in\Sigma\setminus l}\big|\partial H(x,y)/\partial y\big|<1$, 
			and
			$\sup_{(x,y)\in\Sigma\setminus l}\big|\partial H(x,y)/\partial x\big|<1$; \color{black} 
			\item the one-dimensional quotient map $f:[-1,1]\setminus\{0\}\rightarrow[-1,1]$ is $C^1$-smooth and satisfies
			$\lim_{x\rightarrow0^-}f(x)=1$,
			$\lim_{x\rightarrow0^+}f(x)=-1$, $-1<f(x)<1$ and
			$f'(x)>\sqrt{2}$ for every $x\in[-1,1]\setminus\{0\}$.
		\end{itemize}
	\end{itemize}
\end{definition}

\begin{figure}[htbp]
	\centering
	\includegraphics[width=14cm]{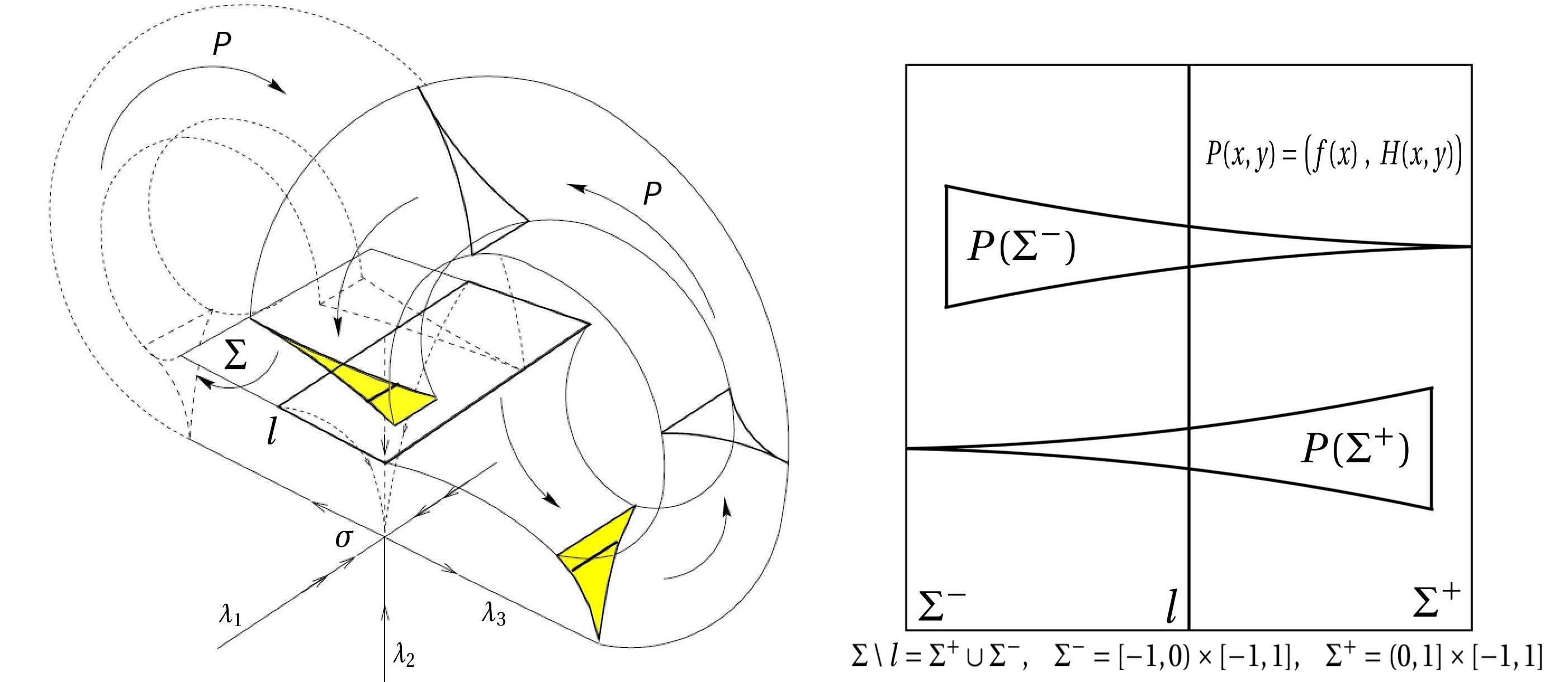}
	\caption{Geometric Lorenz attractor and return map}
\end{figure}
It has been proved that geometric Lorenz attractor is a homoclinic class~\cite[Theorem 6.8]{AP} and  $C^2$-robust~\cite[Proposition 4.7]{STW} .

\begin{proposition}\label{Pro:lorenz-robust}
	Let $r\in\mathbb{N}_{\geq2}\cup\{\infty\}$  and $X\in\mathscr{X}^r(M^3)$. If $X$ exhibits a geometric Lorenz attractor $\Lambda$ with attracting region $U$, then there exists a $C^r$-neighborhood $\cU$ of $X$ in $\mathscr{X}^r(M^3)$, such that for every $Y\in\cU$, $U$ is an attracting region of $Y$, and the maximal invariant set $\Lambda_Y=\bigcap_{t>0}\phi_t^Y(U)$ is a geometric Lorenz attractor. 
	Moreover, the geome\-tric Lorenz attractor is a singular hyperbolic homoclinic class, and every pair of periodic orbits are homoclinic related.
\end{proposition}
For singular hyperbolic attractors, we have the following result.
\begin{theoremalph}\label{Thm:irregular-Lorenz}
	There exists a Baire residual subset $\mathcal{R}^r\subset\mathscr{X}^r(M^3),(r\in\mathbb{N}_{\geq 2})$ and a Baire residual subset $\mathcal{R}\subset\mathscr{X}^1(M)$ so that, 
	if $\Lambda$ is a geometric Lorenz attractor of $X\in\mathcal{R}^r$ or a singular hyperbolic attractor of $X\in\mathcal{R},$ then for any continuous function $g,$ $h$ on $\Lambda$  we have
    \begin{description}
    	\item[(I)] For  any $\alpha\in \mathrm{Int}(L_{g})$ and $0\leq c< \max\{h_{\mu}(\Phi):\mu\in M_{g}(\alpha)\},$ the set $\{\mu\in M_{g}^{erg}(\alpha):h_{\mu}(\Phi)=c,\ S_\mu=\Lambda\}$ is residual in $\{\mu\in M_{g}(\alpha):h_{\mu}(\Phi)\geq c\}.$ 
    	\item[(II)] For  any $\alpha\in \mathrm{Int}(L_{g,h}),$  the set $\{\mu\in M_{g,h}^{erg}(\alpha):S_\mu=\Lambda\}$ is residual in $M_{g,h}(\alpha).$
    \end{description}
\end{theoremalph}
\begin{remark}\label{remark-AA}
	We will prove one general result Theorem~\ref{thm-almost}  stating that when $\Lambda$ is a singular hyperbolic homoclinic class such that each pair of periodic orbits are homoclinically related and $\mathcal{M}(\Phi,\Lambda)=\overline{\mathcal{M}_1(\Lambda)}=\overline{\mathcal{M}_{per}(\Lambda)}$, then the conclusions of Theorem~\ref{Thm:irregular-Lorenz} hold. Then Theorem~\ref{Thm:irregular-Lorenz}  is a direct consequence of Theorem~\ref{thm-almost}. The reason is that when 
	$\Lambda$ is a Lorenz attractor of vector fields in a Baire residual subset $\mathcal{R}^r\subset\mathcal{X}^r(M^3), (r\in\mathbb{N}_{\geq 2})$ or $\Lambda$ is a singular hyperbolic attractor $\Lambda$ of vector fields in a Baire residual set $\mathcal{R}\subset\mathcal{X}^1(M)$,
	then $\Lambda$ is a homoclinic class such that each pair of periodic orbits are homoclinically related (cf.~\cite[Theorem 6.8]{AP} for Lorenz attractors and~\cite[Theorem B]{CY-Robust-attractor} for singular hyperbolic attractors) and $\mathcal{M}(\Phi,\Lambda)=\overline{\mathcal{M}_1(\Lambda)}=\overline{\mathcal{M}_{per}(\Lambda)}$ (cf.~\cite[Theorem A \& B]{STW}). 
\end{remark}

\subsection{Proof of Theorem \ref{Thm:irregular-Lorenz}}
Given  a vector field $X\in\mathscr{X}^1(M)$ and an invariant compact  set $\Lambda$.  Let $\mathcal{N}\subset \mathcal{M}(\Phi,\Lambda)$ be a convex set. We denote $\mathcal{N}_{erg}=\mathcal{M}_{erg}(\Phi,\Lambda)\cap \mathcal{N}.$
Let $d \in \mathbb{N}$ and  $(A, B) \in A(\Phi,\Lambda)^{d} \times A(\Phi,\Lambda)^{d}$ such that $B$ satisfies \eqref{equation-O}.
Denote
$$L_{A,B}^{\mathcal{N}}=\{\mathcal{P}_{A,B}(\mu):\mu\in \mathcal{N}\}.$$
For any $\alpha\in  L_{A,B}^{\mathcal{N}},$ denote  $$M_{A,B}^{\mathcal{N}}(\alpha)=\{\mu\in {\mathcal{N}}:\mathcal{P}_{A,B}(\mu)=\alpha\},\ M_{A,B}^{erg,\mathcal{N}}(\alpha)=\{\mu\in \mathcal{N}_{erg}:\mathcal{P}_{A,B}(\mu)=\alpha\}.$$
Let $\chi:\mathcal{M}(\Phi,\Lambda)\to \mathbb{R}$ be a continuous function. 
We define the pressure of $\chi$ with respect to $\mu$ by $P(\Phi,\chi,\mu)=h_\mu(\Phi)+\chi(\mu).$ 
For any $\alpha\in  L_{A,B}^\mathcal{N},$ denote $$H_{A,B}^\mathcal{N}(\Phi,\chi,\alpha)=\sup\{P(\phi,\chi,\mu):\mu\in M_{A,B}^\mathcal{N}(\alpha)\}.$$ In particular, when $\chi\equiv0,$ we write
$$H_{A,B}^\mathcal{N}(\Phi,\alpha)=H_{A,B}^\mathcal{N}(\Phi,0,\alpha)=\sup\{h_\mu(\Phi):\mu\in M_{A,B}^\mathcal{N}(\alpha)\}.$$

If we replace $\mathcal{M}(\Phi,\Lambda)$ by $\mathcal{N}$ in Section \ref{section-abstract conditions}-\ref{section-proof}, the results of Theorem \ref{thm-Almost-Additive}, Theorem \ref{thm-Almost-Additive2} and Remark \ref{remark-AB} also hold. In fact, the convexity of $\mathcal{M}(\Phi,\Lambda)$ is one core property in the proof  of Theorem \ref{thm-Almost-Additive}, Theorem \ref{thm-Almost-Additive2} and Remark \ref{remark-AB}. Since $\mathcal{N}$ is convex, then the arguments of Section \ref{section-abstract conditions} are also true. Here, we omit the proof.
\begin{theorem}\label{thm-Almost-Additive3}
	Given  a vector field $X\in\mathscr{X}^1(M)$ and an invariant compact  set $\Lambda$.  Assume that the metric entropy $\mathcal{M}(\Phi,\Lambda) \ni \mu \mapsto h_\mu(\Phi)$ is upper semi-continuous.   Let $d \in \mathbb{N}$ and $(A, B) \in A(\Phi,\Lambda)^{d} \times A(\Phi,\Lambda)^{d}$ 
	such that $B$ satisfies \eqref{equation-O}.  Let $\mathcal{N}\subset \mathcal{M}(\Phi,\Lambda)$ be a convex set. Assume that the following holds: 
	for any $F=\cov\{\mu_i\}_{i=1}^m\subseteq \mathcal{N},$ and any $\eta,  \zeta>0$,   there are compact invariant subsets $\Lambda_i\subseteq\Theta\subsetneq \Lambda$ such that for each $i\in\{1,2,\cdots,m\}$
	\begin{description}
		\item[(1)] for any $a\in \mathrm{Int}(\mathcal{P}_{A,B}(\mathcal{M}(\Phi,\Theta)))$ and any $\varepsilon>0,$ there exists an ergodic measure $\mu_a\in\mathcal{N}$ supported on $\Theta$ with $\mathcal{P}_{A,B}(\mu_a)=a$
		such that $|h_{\mu_a}(\Phi)-H(\Phi,a,\Theta)|<\varepsilon,$ where $H(\Phi,a,\Theta)=\sup\{h_\mu(\Phi):\mu\in \mathcal{M}(\Phi,\Theta) \text{ and }\mathcal{P}_{A,B}(\mu)=a\}.$
		\item[(2)] $\htop(\Lambda_i)>h_{\mu_i}(\Phi)-\eta.$
		\item[(3)] $d_H(K,  \mathcal{M}(\Phi,  \Theta))<\zeta$,   $d_H(\mu_i,  \mathcal{M}(\Phi,  \Lambda_i))<\zeta.$
	\end{description}
    Let $\chi:\mathcal{M}(\Phi,\Lambda)\to \mathbb{R}$ be a continuous function. Then 
	\begin{description}
		\item[(I)] For any $\alpha\in \mathrm{Int}(L_{A,B}^\mathcal{N}),$ any $\mu_0\in M_{A,B}^\mathcal{N}(\alpha)$ and any $\eta,  \zeta>0$, there is $\nu\in M_{A,B}^{erg,\mathcal{N}}(\alpha)$ such that $d^*(\nu,\mu_0)<\zeta$ and $|P(\Phi,\chi,\nu)-P(\Phi,\chi,\mu_0)|<\eta.$
	\end{description}
    If further $\{\mu\in \mathcal{N}:h_{\mu}(\Phi)=0\}$ is dense in $\mathcal{N},$ and $\chi:\mathcal{M}(\Phi,\Lambda)\to \mathbb{R}$ satisfies \eqref{equation-W} and \eqref{equation-AF}, then we have
    \begin{description}
    	\item[(II)] 
    	For any $\alpha\in \mathrm{Int}(L_{A,B}^\mathcal{N}),$ any $\mu_0\in M_{A,B}^\mathcal{N}(\alpha),$ any $\max_{\mu\in M_{A,B}^\mathcal{N}(\alpha)}\chi(\mu)\leq c\leq P(\phi,\chi,\mu_0)$ and any $\eta,  \zeta>0$, there is $\nu\in M_{A,B}^{erg,\mathcal{N}}(\alpha)$ such that $d^*(\nu,\mu_0)<\zeta$ and $|P(\Phi,\chi,\nu)-c|<\eta.$ 
    	\item[(III)]  For any $\alpha\in \mathrm{Int}(L_{A,B}^\mathcal{N})$ and $\max_{\mu\in M_{A,B}^\mathcal{N}(\alpha)}\chi(\mu)\leq c< H_{A,B}^\mathcal{N}(\Phi,\chi,\alpha),$ the set $\{\mu\in M_{A,B}^{erg,\mathcal{N}}(\alpha):P(\Phi,\chi,\mu)=c\}$ is residual in $\{\mu\in M_{A,B}^\mathcal{N}(\alpha):P(\Phi,\chi,\mu)\geq c\}.$ If further there is an invariant measure $\tilde{\mu}\in\mathcal{N}$ with $S_{\tilde{\mu}}=\Lambda$, then for any $\alpha\in \mathrm{Int}(L_{A,B}^\mathcal{N})$ and $\max_{\mu\in M_{A,B}^\mathcal{N}(\alpha)}\chi(\mu)\leq c< H_{A,B}^\mathcal{N}(\Phi,\chi,\alpha),$ the set $\{\mu\in M_{A,B}^{erg,\mathcal{N}}(\alpha):P(\Phi,\chi,\mu)=c,\ S_\mu=\Lambda\}$ is residual in $\{\mu\in M_{A,B}^\mathcal{N}(\alpha):P(\Phi,\chi,\mu)\geq c\}.$
    	\item[(IV)] The set $\{(\mathcal{P}_{A,B}(\mu), P(\Phi,\chi,\mu)):\mu\in \mathcal{N},\ \alpha=\mathcal{P}_{A,B}(\mu)\in\mathrm{Int}(L_{A,B}^\mathcal{N}),\ \max_{\mu\in M_{A,B}^\mathcal{N}(\alpha)}\chi(\mu)
    	\leq P(\Phi,\chi,\mu)< H_{A,B}^\mathcal{N}(\Phi,\chi,\alpha)\}$ coincides with $\{(\mathcal{P}_{A,B}(\mu), P(\Phi,\chi,\mu)):\mu\in \mathcal{N}_{erg},\ \alpha
    	=\mathcal{P}_{A,B}(\mu)\in\mathrm{Int}(L_{A,B}^\mathcal{N}), \max_{\mu\in M_{A,B}^\mathcal{N}(\alpha)}\chi(\mu)\leq P(\Phi,\chi,\mu)< H_{A,B}^\mathcal{N}(\Phi,\chi,\alpha)\}.$
    \end{description} 
\end{theorem}

Denote by $\sing(\Lambda)$ the set of singularities for the vector field $X$  in $\Lambda$, by $\mathcal{M}_{per}(\Lambda)$ the set of periodic measures supported on $\Lambda$ and set   $$\mathcal{M}_1(\Lambda)=\left\{\mu\in \mathcal{M}(\Phi,\Lambda): \mu(\sing(\Lambda))=0 \right\} ~~\text{and}~~ \mathcal{M}_0(\Lambda)=\mathcal{M}_{erg}(\Phi,\Lambda)\cap \mathcal{M}_1(\Lambda).$$ 
\begin{proposition}\label{Pro:Horseshoe-approximation-inv2}\cite[Proposition 4.12]{STVW}
	%\marginpar{\color{magenta}\tiny Attempt for strong approximation on the closure of $\mathcal M_1$}
	Let $X\in\mathscr{X}^1(M)$ and $\Lambda$ be a singular hyperbolic homoclinic class of $X$. 
	Assume each pair of periodic orbits of $\Lambda$ are homoclinic related.
	Then  for each $\varepsilon>0$ and any $\mu\in \mathcal{M}_1(\Lambda)$, 
	there exist a basic set $\Lambda'\subset\Lambda$ and $\nu\in\mathcal{M}_{erg}(\Phi,\Lambda')$ so that
	\begin{equation*}
		d^*(\nu,\mu)<\varepsilon ~~~\text{and}~~~h_{\nu}(\Phi)>h_{\mu}(\Phi)-\varepsilon.
	\end{equation*}
\end{proposition}

\begin{proposition}\label{Pro:Horseshoe-approximation-inv}\cite[Proposition 4.13]{STVW}
	%\marginpar{\color{magenta}\tiny Attempt for strong approximation on the closure of $\mathcal M_1$}
	Let $X\in\mathscr{X}^1(M)$ and $\Lambda$ be a singular hyperbolic homoclinic class of $X$. 
	Assume each pair of periodic orbits of $\Lambda$ are homoclinic related, and $\mathcal{M}(\Phi,\Lambda)=\overline{\mathcal{M}_1(\Lambda)}$.  
	Then  for each $\varepsilon>0$ and any $\mu\in \mathcal{M}(\Phi,\Lambda)$, 
	there exist a basic set $\Lambda'\subset\Lambda$ and $\nu\in\mathcal{M}_{erg}(\Phi,\Lambda')$ so that
	\begin{equation*}
		d^*(\nu,\mu)<\varepsilon ~~~\text{and}~~~h_{\nu}(\Phi)>h_{\mu}(\Phi)-\varepsilon.
	\end{equation*}
\end{proposition}
\begin{lemma}\label{Lem:large-horseshoe}\cite[Lemma 4.5]{STVW}
	Let $\Lambda_1$ and $\Lambda_2$ be two basic sets of $X\in\mathcal{X}^1(M)$. Assume there exists hyperbolic periodic points $p_1\in\Lambda_1$ and $p_2\in\Lambda_2$ such that $\orb(p_1)$ and $\orb(p_2)$ are homoclinically related. Then there exists a larger basic set $\Lambda$ that contains both $\Lambda_1$ and $\Lambda_2$.
\end{lemma}
Now we state a result on 'multi-horseshoe' dense property of singular hyperbolic homoclinic classes.
\begin{theorem}\label{Mainlemma-convex-by-horseshoe}
	Let $X\in\mathscr{X}^1(M)$ and $\Lambda$ be a singular hyperbolic homoclinic class of $X$. 
	Assume each pair of periodic orbits of $\Lambda$ are homoclinic related.    Then $\Lambda$ satisfies the 'multi-horseshoe' dense property on $\mathcal{M}_1(\Lambda)$. Moreover, if $\mathcal{M}(\Phi,\Lambda)=\overline{\mathcal{M}_1(\Lambda)}=\overline{\mathcal{M}_{per}(\Lambda)}$.    Then $\Lambda$ satisfies the 'multi-horseshoe' dense property.
\end{theorem}
\begin{proof}
	Fix $F=\cov\{\mu_i\}_{i=1}^m\subseteq \mathcal{M}_1(\Lambda)(\mathcal{M}(\Phi,\Lambda)),$ and $\eta,  \zeta>0.$ By Proposition \ref{Pro:Horseshoe-approximation-inv2} (Proposition \ref{Pro:Horseshoe-approximation-inv} ) for each $1\leq i\leq m$ there exist a basic set $\Lambda'_i\subset\Lambda$ and $\nu_i\in\mathcal{M}_{erg}(\Phi,\Lambda'_i)$ so that
	\begin{equation*}
		d^*(\nu_i,\mu_i)<\frac{\zeta}{2} ~~~\text{and}~~~h_{\nu_i}(\Phi)>h_{\mu_i}(\Phi)-\frac{\eta}{2}.
	\end{equation*}
	By Lemma \ref{Lem:large-horseshoe} there exists a larger basic set $\tilde{\Lambda}$ that contains every $\Lambda'_i$.
	Apply Theorem \ref{Mainlemma-convex-by-horseshoe5} to $\tilde{\Lambda},$ $\tilde{F}=\cov\{\nu_i\}_{i=1}^m\subseteq \mathcal{M}(\Phi,\tilde{\Lambda})$ and $\frac{\eta}{2},$ $\frac{\zeta}{2},$ we complete the proof.
\end{proof}
Now we show that the results of Theorem \ref{thm-Almost-Additive3} hold for singular hyperbolic homoclinic classes.
\begin{theorem}\label{thm-almost3}
	Let $X\in\mathscr{X}^1(M)$ and $\Lambda$ be a singular hyperbolic homoclinic class of $X$. 
	Assume each pair of periodic orbits of $\Lambda$ are homoclinic related.  Let $\chi:\mathcal{M}(\Phi,\Lambda)\to \mathbb{R}$ be a continuous function satisfying \eqref{equation-W} and \eqref{equation-AF}.  Let $d \in \mathbb{N}$ and $(A, B) \in A(\Phi,\Lambda)^{d} \times A(\Phi,\Lambda)^{d}$ satisfying \eqref{equation-UV}. 
	Then:
	\begin{description}
		\item[(I)] For any $\alpha\in \mathrm{Int}(L_{A,B}^{\mathcal{M}_1(\Lambda)}),$ any $\mu_0\in M_{A,B}^{\mathcal{M}_1(\Lambda)}(\alpha),$ any $\max_{\mu\in M_{A,B}^{\mathcal{M}_1(\Lambda)}(\alpha)}\chi(\mu)\leq c\leq P(\phi,\chi,\mu_0)$ and any $\eta,  \zeta>0$, there is $\nu\in M_{A,B}^{erg,{\mathcal{M}_1(\Lambda)}}(\alpha)$ such that $d^*(\nu,\mu_0)<\zeta$ and $|P(\Phi,\chi,\nu)-c|<\eta.$ 
		\item[(II)] For any $\alpha\in \mathrm{Int}(L_{A,B}^{\mathcal{M}_1(\Lambda)})$ and $\max_{\mu\in M_{A,B}^{\mathcal{M}_1(\Lambda)}(\alpha)}\chi(\mu)\leq c< H_{A,B}^{\mathcal{M}_1(\Lambda)}(\Phi,\chi,\alpha),$ the set $\{\mu\in M_{A,B}^{erg,{\mathcal{M}_1(\Lambda)}}(\alpha):P(\Phi,\chi,\mu)=c\}$ is residual in $\{\mu\in M_{A,B}^{\mathcal{M}_1(\Lambda)}(\alpha):P(\Phi,\chi,\mu)\geq c\}.$
		\item[(III)] $\{(\mathcal{P}_{A,B}(\mu), P(\Phi,\chi,\mu)):\mu\in {\mathcal{M}_1(\Lambda)},\ \alpha=\mathcal{P}_{A,B}(\mu)\in\mathrm{Int}(L_{A,B}^{\mathcal{M}_1(\Lambda)}),\ \max_{\mu\in M_{A,B}^{\mathcal{M}_1(\Lambda)}(\alpha)}\chi(\mu)\leq P(\Phi,\chi,\mu)< H_{A,B}^{\mathcal{M}_1(\Lambda)}(\Phi,\chi,\alpha)\}$ coincides with $\{(\mathcal{P}_{A,B}(\mu), P(\Phi,\chi,\mu)):\mu\in {\mathcal{M}_0(\Lambda)},\ \alpha=\mathcal{P}_{A,B}(\mu)\in\mathrm{Int}(L_{A,B}^{\mathcal{M}_1(\Lambda)}), \max_{\mu\in M_{A,B}^{\mathcal{M}_1(\Lambda)}(\alpha)}\chi(\mu)\leq P(\Phi,\chi,\mu)< H_{A,B}^{\mathcal{M}_1(\Lambda)}(\Phi,\chi,\alpha)\}.$  
	\end{description}
\end{theorem}
\begin{proof}
	Note that since $\Lambda$ is a singular hyperbolic homoclinic class, the vector field $X$ satisfies the star condition in a neighborhood of $\Lambda$. More precisely, there exist a neighborhood $U$ of $\Lambda$ and a $C^1$-neighborhood $\mathcal{U}$ of $X$ in $\mathscr{X}^1(M)$ such that every critical element contained in $U$ associated to a vector field $Y\in\mathcal{U}$  is hyperbolic. Then by Lemma \ref{lemma-semicontinuous} 
	the entropy function is upper semi-continuous. 
	Note that $\mathcal{M}(\Phi,\Theta)\subset \mathcal{M}_1(\Lambda)$ if $\Theta\subset \Lambda$ is a horseshoe.
	Since the metric entropy of periodic measure is zero, then $\{\mu\in  \mathcal{M}_1(\Lambda):h_{\mu}(\Phi)=0\}$ is dense in $ \mathcal{M}_1(\Lambda)$ by Theorem \ref{Mainlemma-convex-by-horseshoe}.
	So we complete the proof by Theorem \ref{condition-principle}, Theorem \ref{thm-Almost-Additive3} and  Theorem \ref{Mainlemma-convex-by-horseshoe}.
\end{proof}
Combining with Theorem \ref{Pro:lorenz-robust}, the results of Theorem \ref{thm-almost3} hold for geometric Lorenz attractor 
\begin{corollary}
	Let $r\in\mathbb{N}_{\geq2}\cup\{\infty\}$  and $X\in\mathscr{X}^r(M^3)$. If $X$ exhibits a geometric Lorenz attractor $\Lambda$, then the results of Theorem \ref{thm-almost3} hold for $\Lambda.$
\end{corollary}

If further we have $\mathcal{M}(\Phi,\Lambda)=\overline{\mathcal{M}_1(\Lambda)}=\overline{\mathcal{M}_{per}(\Lambda)}$,   then using Theorem \ref{condition-principle}, Theorem \ref{thm-Almost-Additive}, Theorem \ref{thm-Almost-Additive2},  Remark \ref{remark-AB} , Theorem \ref{Mainlemma-convex-by-horseshoe} and Lemma \ref{lemma-full-support},  we have the following result for singular hyperbolic homoclinic classes.
\begin{theorem}\label{thm-almost}
	Let $X\in\mathscr{X}^1(M)$ and $\Lambda$ be a singular hyperbolic homoclinic class of $X$. 
	Assume each pair of periodic orbits of $\Lambda$ are homoclinic related, and $\mathcal{M}(\Phi,\Lambda)=\overline{\mathcal{M}_1(\Lambda)}=\overline{\mathcal{M}_{per}(\Lambda)}$.  Let $\chi:\mathcal{M}(\Phi,\Lambda)\to \mathbb{R}$ be a continuous function satisfying \eqref{equation-W} and \eqref{equation-AF}.  Let $d \in \mathbb{N}$ and $(A, B) \in A(\Phi,\Lambda)^{d} \times A(\Phi,\Lambda)^{d}$ satisfying \eqref{equation-UV}. 
	Then:
	\begin{description}
		\item[(I)] For any $\alpha\in \mathrm{Int}(L_{A,B}),$ any $\mu_0\in M_{A,B}(\alpha),$ any $\max_{\mu\in M_{A,B}(\alpha)}\chi(\mu)\leq c\leq P(\phi,\chi,\mu_0)$ and any $\eta,  \zeta>0$, there is $\nu\in M_{A,B}^{erg}(\alpha)$ such that $d^*(\nu,\mu_0)<\zeta$ and $|P(\Phi,\chi,\nu)-c|<\eta.$ 
		\item[(II)] For any $\alpha\in \mathrm{Int}(L_{A,B})$ and $\max_{\mu\in M_{A,B}(\alpha)}\chi(\mu)\leq c< H_{A,B}(\Phi,\chi,\alpha),$ the set $\{\mu\in M_{A,B}^{erg}(\alpha):P(\Phi,\chi,\mu)=c,\ S_\mu=\Lambda\}$ is residual in $\{\mu\in M_{A,B}(\alpha):P(\Phi,\chi,\mu)\geq c\}.$
		\item[(III)] $\{(\mathcal{P}_{A,B}(\mu), P(\Phi,\chi,\mu)):\mu\in \mathcal{M}(\Phi,M),\ \alpha=\mathcal{P}_{A,B}(\mu)\in\mathrm{Int}(L_{A,B}),\ \max_{\mu\in M_{A,B}(\alpha)}\chi(\mu)\leq P(\Phi,\chi,\mu)< H_{A,B}(\Phi,\chi,\alpha)\}$ coincides with $\{(\mathcal{P}_{A,B}(\mu), P(\Phi,\chi,\mu)):\mu\in \mathcal{M}_{erg}(\Phi,M),\ \alpha=\mathcal{P}_{A,B}(\mu)\in\mathrm{Int}(L_{A,B}), \max_{\mu\in M_{A,B}(\alpha)}\chi(\mu)\leq P(\Phi,\chi,\mu)< H_{A,B}(\Phi,\chi,\alpha)\}.$  
	\end{description}
\end{theorem}

Proceeding in a similar manner to Theorem \ref{Thm:basic-set1} and \ref{Thm:basic-set2},
let $\chi\equiv 0$ and $d=1,2$ in  Theorem \ref{thm-almost},  we obtain  Theorem \ref{Thm:irregular-Lorenz} by Remark \ref{remark-AA}.

\subsection*{Acknowledgments}
 The authors are
 supported by the National Natural Science Foundation of China (grant No. 12071082, 11790273) and in part by Shanghai Science and Technology Research Program (grant No. 21JC1400700).
 
\subsection*{Declarations}
{\bf  Conﬂict of interest} The authors declare that they have no conﬂict of interest.

\bibliographystyle{plain}

\end{document}